\documentclass[11pt]{Andrea}
\usepackage{amsfonts, amssymb, amsmath, amsthm, latexsym, enumerate, epsfig, color, mathrsfs, fancyhdr, pdfsync, eurosym, endnotes, stmaryrd}
\usepackage[all]{xy}

\usepackage[english]{babel}

\usepackage[mac]{inputenc}
\usepackage[T1]{fontenc}

\usepackage{hyperref}

\newcommand{\Z}{\mathbf{Z}}
\newcommand{\R}{\mathbf{R}}
\newcommand{\A}{\mathbf{A}}
\renewcommand{\P}{\mathbf{P}}

\newcommand{\As}{\mathscr{A}}
\newcommand{\Bs}{\mathscr{B}}
\newcommand{\Cs}{\mathscr{C}}
\newcommand{\Ds}{\mathscr{D}}
\newcommand{\Es}{\mathscr{E}}
\newcommand{\Fs}{\mathscr{F}}
\newcommand{\Gs}{\mathscr{G}}
\newcommand{\Hs}{\mathscr{H}}
\newcommand{\Js}{\mathscr{J}}
\newcommand{\Os}{\mathscr{O}}
\newcommand{\Vs}{\mathscr{V}}
\newcommand{\Xs}{\mathscr{X}}

\newcommand{\Mc}{\mathcal{M}}
\newcommand{\Rc}{\mathcal{R}}
\newcommand{\bRc}{\boldsymbol{\mathcal{R}}}

\newcommand{\Xk}{\mathfrak{X}}

\newcommand{\E}[2]{\mathbf{A}^{#1,\mathrm{an}}_{#2}}
\newcommand\ti[1]{\tilde{#1}}
\newcommand\wti[1]{\widetilde{#1}}

\newcommand{\ho}{\hat{\otimes}}
\newcommand{\an}{\textrm{an}}

\newcommand{\cn}[2]{[\![ {#1},{#2}]\!]}

\newcommand\wc{{\mkern 2mu\cdot\mkern 2mu}}
\DeclareMathOperator\rk{rk}

\theoremstyle{plain}
\newtheorem{theorem}{Theorem}[subsection]
\newtheorem{lemma}[theorem]{Lemma}
\newtheorem{proposition}[theorem]{Proposition}
\newtheorem{corollary}[theorem]{Corollary}

\theoremstyle{definition}
\newtheorem{definition}[theorem]{Definition}
\newtheorem{notation}[theorem]{Notation}
\newtheorem{setting}{Setting}[section]

\theoremstyle{remark}
\newtheorem{remark}[theorem]{Remark}

\numberwithin{equation}{subsection}

\begin{document}

\title[The convergence Newton polygon on curves]{The convergence Newton polygon\\ of a $p$-adic differential equation~II:\\ Continuity and finiteness on Berkovich curves}
\author{J\'er\^ome Poineau}
\email{jerome.poineau@unicaen.fr}
\address{Laboratoire de math\'ematiques Nicolas Oresme, Université de Caen, BP 5186, F-14032 Caen Cedex}

\author{Andrea Pulita}
\email{pulita@math.univ-montp2.fr}
\address{D\'epartement de math\'ematiques, Universit\'e de Montpellier II, CC051, Place Eug\`ene Bataillon, F-34095 Montpellier Cedex 5}
\date{\today}


\begin{abstract}
We study the variation of the convergence Newton polygon of a differential equation along a smooth Berkovich curve over a non-archimedean complete valued field of characteristic~0. Relying on work of the second author who investigated its properties on affinoid domains of the affine line, we prove that its slopes give rise to continuous functions that factorise by the retraction through a locally finite subgraph of the curve. 
\end{abstract}

\maketitle

\section{Introduction}

In the $p$-adic setting, linear differential equations exhibit some features that do not appear over the complex numbers. For instance, even if the coefficients of such an equation are entire functions, it may fail to have a solution that converges everywhere. This leads to a non-trivial notion of radius of convergence, which has been actively investigated and found several applications. In this article, we study the variation of the radius of convergence (or more accurately of the slopes of the convergence Newton polygon, which also contains more refined radii) of a module with connection over a curve and prove that it enjoys continuity and finiteness properties.

The precise setting is the following. Let~$K$ be a non-archimedean complete valued field of characteristic 0. Let~$X$ be a quasi-smooth $K$-analytic curve, in the sense of Berkovich theory. Let~$\Fs$ be a locally free $\Os_{X}$-module of finite type endowed with an integrable connection~$\nabla$.

If~$X$ can be embedded into the affine line~$\A^{1,\an}_{K}$, we have a coordinate on~$X$ that we may use to define the radius of convergence. This study has been carried out by the second author in~\cite{finiteness}, to which this article is a follow-up.

Let us now return to the case of a general quasi-smooth curve~$X$. Let~$x$ be a $K$-rational point of~$X$. By the implicit function theorem, in the neighbourhood of the point~$x$, the curve is isomorphic to a disc, and we may consider the radius of the biggest disc on which a given horizontal section of~$(\Fs,\nabla)$ (\textit{i.e.} a given solution of the associated differential equation) converges. Considering the radii associated to the various horizontal sections, we define a tuple~$\bRc_{S}(x)$ that we call multiradius of convergence at~$x$. (It depends on an additional datum~$S$ on the curve, to which we shall come back later.) For readers who are familiar with the notion of convergence Newton polygon, let us mention that we recover here the tuple of its slopes, up to a logarithm. Since any point of the curve may be changed into a rational one by a suitable extension of the scalars, we may actually extend the definition of the multiradius of convergence to the whole curve. 

It is already interesting to consider the radius of convergence~$\Rc_{S,1}(x)$ at a point~$x$, \textit{i.e.} the radius of the biggest disc on which~$(\Fs,\nabla)$ is trivial, or, equivalently, the minimal radius that appears in the multiradius. In~\cite{ContinuityCurves}, F.~Baldassarri proved the continuity of this function. 

In the article~\cite{finiteness}, using different methods, the second author thoroughly investigated the multiradius of convergence~$\bRc_{S}$ in the case where~$X$ is an affinoid domain of the affine line. He proved that this function is continuous, piecewise $\log$-linear and satisfies a strong finiteness property: it factorises by the retraction through a finite subgraph of~$X$. 

In this paper, we prove that those properties extend to arbitrary quasi-smooth $K$-analytic curves:

{
\renewcommand{\thetheorem}{\ref{thm:continuousandfinite}}
\begin{theorem}
The map~$\bRc_{S}$ satisfies the following properties:
\begin{enumerate}[\it i)]
\item it is continuous;
\item its restriction to any locally finite graph~$\Gamma$ is piecewise $\log$-linear;
\item on any interval~$J$, the $\log$-slopes of its restriction are rational numbers of the form~$m/j$ with $m\in\Z$ and $j\in\cn{1}{r}$, where~$r$ is the rank of~$\Fs$ around~$J$; 
\item there exists a locally finite subgraph~$\Gamma$ of~$X$ and a continuous retraction $r\colon X \to \Gamma$ such that the map~$\bRc_{S}$ factorises by~$r$.
\end{enumerate}
\end{theorem}
\addtocounter{theorem}{-1}
}

This theorem lays the foundations of a deeper study of $p$-adic differential equations on curves. In subsequent work, we will use it in a crucial way to prove decomposition theorems (see~\cite{finiteness3}) and finiteness results for the de Rham cohomology of modules with connections (see~\cite{finiteness4}). We mention that F.~Baldassarri and K.~Kedlaya have announced similar results (see~\cite{Kedlayalocalglobal}).

\medbreak

Let us now give an overview of the contents of the article. We have just explained the rough idea underlying the definition of the multiradius of convergence on the curve~$X$. Obviously, some work remains to be done in order to put together the different local normalisations in a coherent way and give a proper definition of this multiradius. This will be the content of our first section. The first definition was actually given by F.~Baldassarri in~\cite{ContinuityCurves}, using a semistable formal model of the curve, which led to some restrictions. Here, we will use A.~Ducros's notion of triangulation, which is a way of cutting a curve into pieces that are isomorphic to virtual discs or annuli. (The existence of triangulations is very closely related to the semistable reduction theorem.) Among other advantages, it lets us deal with arbitrary quasi-smooth curve (regardless of the fact that they are compact or not, strictly analytic or not, over a field that is trivially valued or not) and enables us to define everything within the realm of analytic spaces. The symbol~$S$ that appeared in our notation~$\bRc_{S}$ for the multiradius above actually referred to the choice of such a triangulation. We will explain how our radius of convergence relates to that of~\cite{ContinuityCurves} and to the usual one in the case of analytic domains of the affine line, as defined in~\cite{ContinuityBDV} or~\cite{finiteness}, for instance. 

As indicated above, extension of scalars plays a crucial role in the theory. We devote \S\ref{sec:extensionscalars} to it and give a precise description of the fibres of the extension maps in the case of curves over algebraically closed fields: they consist of one universal point together with a bunch of virtual open discs that spread out of it. We believe those results to be of independent interest.

Finally, \S\ref{sec:result} is devoted to the proof of the continuity, $\log$-linearity and finiteness property of the multiradius of convergence~$\bRc_{S}$. The basic idea is to present the curve~$X$ locally as a finite \'etale cover of an affinoid domain~$W$ of the affine line, apply the results of~\cite{finiteness} to~$W$ and pull them back. We will first present the geometric results we need (which essentially come from A.~Ducros's manuscript~\cite{RSSen}) to find a nice presentation of the curve, then prove a weak version of the finiteness property (local constancy of~$\bRc_{S}$ outside a locally finite subgraph) and finally the continuity (which will yield the stronger finiteness property) and piecewise $\log$-linearity.

\begin{setting}\label{setting}
For the rest of the article, we fix the following: $K$~is a complete valued field of characteristic~0, $p$~is the characteristic exponent of its residue field~$\ti K$ (either~1 or a prime number), $X$~is a quasi-smooth $K$-analytic curve\footnote{Quasi-smooth means that $\Omega_{X}$ is locally free of rank~1, see~\cite[3.1.11]{RSSen}. This corresponds to the notion called ``rig-smooth'' in the rigid analytic terminology.}, $\Fs$~is a locally free $\Os_{X}$-module of finite type endowed with an integrable connection~$\nabla$. We fix an algebraic closure~$\bar{K}$ of~$K$ and denote its completion by~$K^a$.

From \S\ref{sec:radius} on, we will assume that the curve~$X$ is endowed with a weak triangulation~$S$. 

From \S\ref{sec:geometry} on, we will assume that the field~$K$ is algebraically closed.
\end{setting}

\section{Definitions}

In this section, we define the radius of convergence of~$(\Fs,\nabla)$ at any point of the curve~$X$. To achieve this task, we will need to understand precisely the geometry of~$X$. Our main tool will be A.~Ducros's notion of triangulation (see~\cite[5.1.13]{RSSen}) that we recall here. The original definition, introduced by F.~Baldassarri in~\cite{ContinuityCurves}, in a more restrictive setting, made use of a semistable model of the curve. The two points of view are actually very close (see \S\ref{sec:Baldassarri}).

Let us point out that the structure of analytic curves has been completely described by V.~Berkovich thanks to the semistable reduction theorem (see~\cite[Chapter~4]{rouge}). In the book~\cite{RSSen}, A.~Ducros recently managed to get the same results working only on the analytic side. As his text is a thorough manuscript with easily quotable results, we have chosen to use it systematically when we needed to refer to results on curves, even in the case when they were known before (as for bases of neighbourhoods of points, for instance).

\subsection{Triangulations}

First of all, let us fix notations for discs and annuli.

\begin{notation}
Let~$\E{1}{K}$ be the affine analytic line with coordinate~$t$. Let~$M$ be a complete valued extension of~$K$ and $c\in M$. For $R>0$, we set 
\[D_{M}^+(c,R) = \big\{x\in \E{1}{M}\, \big|\, |(t-c)(x)|\le R\big\}\] 
and  
\[D_{M}^-(c,R) = \big\{x\in \E{1}{M}\, \big|\, |(t-c)(x)|<R\big\}.\]
Denote by~$x_{c,R}$ the unique point of the Shilov boundary of~$D_{M}^+(c,R)$.

For $R_{1},R_{2}$ such that $0 < R_{1} \le R_{2}$, we set
\[C_{M}^+(c,R_{1},R_{2}) = \big\{x\in \E{1}{M}\, \big|\, R_{1}\le |(t-c)(x)|\le R_{2}\big\}.\] 
For $R_{1},R_{2}$ such that $0 < R_{1} < R_{2}$, we set
\[C_{M}^-(c,R_{1},R_{2}) = \big\{x\in \E{1}{M}\, \big|\, R_{1} < |(t-c)(x)| < R_{2}\big\}.\] 
We may suppress the index~$M$ when it is obvious from the context.
\end{notation}

Let us now recall that a non-empty connected $K$-analytic space is called a virtual disc (resp. annulus) if it becomes isomorphic to a union of discs (resp. annuli whose orientations are preserved by Gal($\bar{K}/K$)) over~$K^a$ (see~\cite[3.6.32 and 3.6.35]{RSSen}). 

The skeleton of a virtual open (resp. closed) annulus~$C$ is the set of points~$\Gamma_{C}$ that have no neighbourhoods isomorphic to a virtual disc. It is homeomorphic to an open (resp. closed) interval and the space~$C$ retracts continuously onto it.

\begin{definition}
A subset~$S$ of~$X$ is said to be a weak triangulation of~$X$ if 
\begin{enumerate}
\item $S$ is locally finite and only contains points of type~2 or~3;
\item any connected component of~$X\setminus S$ is a virtual open disc or annulus.
\end{enumerate}

The union of~$S$ and of the skeletons of the connected components of $X\setminus S$ that are virtual annuli forms a locally finite graph, which is called the skeleton~$\Gamma_{S}$ of the weak triangulation~$S$.
\end{definition}

\begin{remark}
The fact that the skeleton~$\Gamma_{S}$ is a locally finite graph is not obvious from the definition. It follows for instance from Theorem~\ref{thm:bonvois} below.
\end{remark}

Usually the skeleton is the union of the segments between the points of~$S$ but beware that peculiar situations may sometimes appear: for example, the skeleton of an open annulus endowed with the empty weak triangulation is an open segment.

Let us also point out that the skeleton of an open disc endowed with the empty weak triangulation is empty, which sometimes forces to handle this case separately.

\begin{remark}\label{rem:retraction}
Let~$Z$ be a connected component of~$X$ such that $Z\cap \Gamma_{S} \ne \emptyset$ (which is always the case except if~$Z$ is an open disc such that~$Z\cap S =\emptyset$). Then we have a natural continuous topologically proper retraction $Z \to Z\cap \Gamma_{S}$. It is the identity on~$Z \cap \Gamma_{S}$ and sends a point of $Z \setminus \Gamma_{S}$ to the unique boundary point of the connected component of $Z \setminus \Gamma_{S}$ (a virtual open disk by definition) containing it.
\end{remark}

\begin{remark}
The definition of a triangulation that is used by A.~Ducros is actually stronger than ours since he requires the connected components of~$X\setminus S$ to be relatively compact. For example, the empty weak triangulation of the open disc is not allowed. 

This property allows him to have a natural continuous retraction $X \to \Gamma_{S}$ and also to associate to~$S$ a nice formal model of the curve~$X$ (see~\cite[\S6.3]{RSSen}, for details). 
\end{remark}

The existence of a triangulation is one of the main results of A.~Ducros's manuscript (see~\cite[Th\'eor\`eme~5.1.14]{RSSen}). It is essentially equivalent to the semistable reduction theorem (see~\cite[\S4]{Triangulations} and~\cite[Chapitre~6]{RSSen}).

\begin{theorem}[(A.~Ducros)]
Any quasi-smooth $K$-analytic curve admits a triangulation, hence a weak triangulation.
\end{theorem}

\subsection{Extension of scalars}\label{sec:extensionscalars}

In this section, we study the effect of extending the scalars on a given weak triangulation. 

\begin{notation}
For every complete valued extension~$L$ of~$K$, we set $X_{L} = X \ho_{K} L$. 

For every complete valued extension~$L$ of~$K$, and every complete valued extension~$M$ of~$L$, we denote by $\pi_{M/L} \colon X_{M} \to X_{L}$ the canonical projection morphism. We set $\pi_{M} = \pi_{M/K}$. 
\end{notation}

Starting from the triangulation~$S$ of~$X$, it is easy to check that the preimage $S_{K^a} = \pi^{-1}_{K^a}(S)$ is a weak triangulation of~$X_{K^a}$.   

To go further, let us introduce the notion of universal point of~$X$: roughly speaking, it is a point that may be canonically lifted to any base field extension~$X_{L}$ of~$X$. We refer to~\cite[\S5.2]{rouge}\footnote{In this reference, where universal points first appeared, they are called ``peaked points''.} and~\cite{Angie} for more information.

\begin{definition}\label{defi:universal}
A point~$x$ of~$X$ is said to be universal if, for any complete valued extension~$L$ of~$K$, the tensor norm on the algebra $\Hs(x)\ho_{K} L$ is multiplicative. In this case, it defines a point of~$X_{L}$ that we denote by~$x_{L}$.
\end{definition}

\begin{lemma}\label{lem:universalmorphism}
Let $\varphi \colon Y \to Z$ be a morphism of $K$-analytic spaces. Let~$y\in Y$ be a universal point. Then~$\varphi(y)$ is universal.

Let~$L$ be a complete valued extension of~$K$ and denote by $\varphi_{L} \colon Y_{L} \to Z_{L}$ the morphism obtained after extension of scalars. We have $\varphi(y)_{L} = \varphi_{L}(y_{L})$. 
\end{lemma}
\begin{proof}
For every complete valued extension~$L$ of~$K$, we have embeddings
\[\Hs(\varphi(y))\ho_{K} L \hookrightarrow \Hs(y)\ho_{K} L \hookrightarrow \Hs(y_{L}).\]
The first map is isometric by~\cite[Lemme~3.1]{Angie}, for instance, and the second by the definition of a universal point. We deduce that the tensor norm on $\Hs(\varphi(y))\ho_{K} L$ is multiplicative. Moreover, we get an isometric embedding $\Hs(\varphi(y)_{L}) \hookrightarrow \Hs(y_{L})$, which shows that $\varphi_{L}(y_{L}) =\varphi(y)_{L}$.
\end{proof}

In some cases, the universality condition is not restrictive:

\begin{theorem}[\protect{\cite[Corollaire~3.14]{Angie}}]\label{thm:universal}
Over an algebraically closed field, every point is universal. 
\end{theorem}

Theorem~\ref{thm:universal} gives a way to lift the given weak triangulation~$S$ of~$X$ to a weak triangulation~$S_{L}$ of~$X_{L}$, for any complete valued extension~$L$ of~$K^a$.

\begin{notation}
Let~$L$ be a complete valued extension of~$K$. We denote by $\mathrm{Gal}^c(L/K)$ the group of isometric automorphisms of~$L$ that induce the identity on~$K$.
\end{notation}

\begin{lemma}\label{lem:Galoisuniversal}
Let~$L$ be a complete valued extension of~$K^a$. Let~$x \in X_{K^a}$ and $\sigma\in \mathrm{Gal}^c(L/K)$. Then, we have
\[\sigma(x)_{L} = \sigma(x_{L}).\]
\end{lemma}
\begin{proof}
The point~$\sigma(x_{L})$ lies over~$\sigma(x)$ in~$X_{K^a}$, hence it belongs to $\Mc(\Hs(\sigma(x)) \ho_{K^a} L)$. By definition of a universal point, as semi-norms on $\Hs(\sigma(x)) \ho_{K^a} L$, we have $\sigma(x_{L}) \le \sigma(x)_{L}$.

The reverse inequality follows from the fact that, as semi-norms on $\Hs(x) \ho_{K^a} L$, we have $\sigma^{-1}(\sigma(x)_{L}) \le x_{L}$, because~$x_{L}$ is universal.

\end{proof}

\begin{definition}
Let~$L$ be a complete valued extension of~$K$. Let~$L^a$ be the completion of an algebraic closure of~$L$, which we see as an extension of~$K^a$. Set
\[S_{L^a} = \{x_{L^a} \mid x\in S_{K^a}\},\ \Gamma_{S_{L^a}} =\{x_{L^a} \mid x\in \Gamma_{S_{K^a}}\}\]
and
\[S_{L} =  \pi_{L^a/L}(S_{L^a}),\ \Gamma_{S_{L}} =  \pi_{L^a/L}(\Gamma_{S_{L^a}}).\]
\end{definition}

By Lemma~\ref{lem:Galoisuniversal}, the sets~$S_{L^a}$ and~$\Gamma_{S_{L^a}}$ are invariant under the action of $\mathrm{Gal}^c(L^a/K)$. We deduce that the sets of the previous definition are well-defined and independent of the choices. We write down the following property for future reference.

\begin{lemma}\label{lem:Sinvar}
Let~$L$ be a complete valued extension of~$K$. The sets~$S_{L}$ and~$\Gamma_{S_{L}}$ are invariant under the action of $\mathrm{Gal}^c(L/K)$.
\end{lemma}

To extend the weak triangulation to the field~$L$, the key point is the following result.

\begin{theorem}\label{thm:structurefibre}
Let~$x$ be a point of~$X_{K^a}$ and let~$L$ be a complete valued extension of~$K^a$. 

\begin{enumerate}
\item If~$x$ has type~$i\in\{1,2\}$, then~$x_{L}$ has type~$i$. If~$x$ has type~$j\in\{3,4\}$, then~$x_{L}$ has type~$j$ or~2.
\item The fibre $\pi_{L/K^a}^{-1}(x)$ is connected and the connected components of~$\pi_{L/K^a}^{-1}(x) \setminus\{x_{L}\}$ are virtual open discs with boundary~$\{x_{L}\}$. Moreover they are open in~$X_{L}$.
\end{enumerate}
\end{theorem}

Before proving the theorem, we will state two consequences.

\begin{corollary}
For any complete valued extension~$L$ of~$K$, the set~$S_{L}$ is a weak triangulation of~$X_{L}$ whose skeleton is~$\Gamma_{S_{L}}$.
\end{corollary}
\begin{proof}
It is enough to prove that~$S_{L^a}$ is a weak triangulation of~$X_{L^a}$ whose skeleton is~$\Gamma_{S_{L^a}}$. By Theorem~\ref{thm:structurefibre} (i), the set~$S_{L^a}$ contains only points of type~2 or~3. The projection~$S_{K^a}$ of~$S_{L^a}$ on~$X_{K^a}$ is locally finite and there is exactly one point of~$S_{L^a}$ above any point of~$S_{K^a}$. We deduce that~$S_{L^a}$ is locally finite.

Let~$x \in S_{K^a}$. Let~$C$ be a connected component of $\pi^{-1}_{L^a/K^a}(x)\setminus\{x_{L^a}\}$. By Theorem~\ref{thm:structurefibre} (ii), it is a disc and it is open in~$X_{L^a}$. Moreover, since $\pi^{-1}_{L^a/K^a}(x)$ is closed in~$X_{L^a}$, $C$ is also closed in~$X_{L^a}\setminus S_{L^a}$. We deduce that the disc~$C$ is a connected component of~$X_{L^a}\setminus S_{L^a}$.

The complement in~$X_{L^a}\setminus S_{L^a}$ of all the connected components of the preceding form is the complement of~$\pi^{-1}_{L^a/K^a}(S_{K^a})$ in~$X_{L^a}$, \textit{i.e.} the union of all sets of the form~$\pi^{-1}_{L^a/K^a}(O)$, where~$O$ is a connected component of~$X_{K^a} \setminus S_{K^a}$. We deduce that the latter are the remaining connected components of~$X_{L^a}\setminus S_{L^a}$.

Let~$O$ be a connected component of~$X_{K^a} \setminus S_{K^a}$. If it is an open disc, then~$\pi^{-1}_{L^a/K^a}(O)$ is an open disc too. If it is an open annulus, then~$\pi^{-1}_{L^a/K^a}(O)$ is an open annulus too and a direct computation shows that its skeleton is the set $\{x_{L^a}, x\in \Gamma_{O}\}$, where~$\Gamma_{O}$ denotes the skeleton of~$O$. The result follows.
\end{proof}

We will now be more precise about the open discs that appear in the fibres of the extension of scalars and show that they are actually isomorphic. In order to do so, we will need results about maximally complete valued fields, for which we refer to~\cite{Poonenmax}.

\begin{theorem}\label{thm:GaloisPoonen}
Let~$M$ be a non-archimedean valued field that is algebraically closed and maximally complete. Let~$L$ be a subfield of~$M$. Let~$L'$ be a valued extension of~$L$ such that there exist a morphism $f \colon \wti{L'} \hookrightarrow \wti M$ between the residue fields and an injective morphism $g \colon |L'^\ast| \hookrightarrow |M^\ast|$ between the value groups that make the following diagrams commute:
\[\xymatrix{
& \wti{L'} \ar[d]^f\\
\wti L \ar[r] \ar[ru] &\wti M
}
\qquad\qquad
\xymatrix{
& |L'^\ast| \ar[d]^g\\
|L^\ast| \ar[r] \ar[ru] & |M^\ast|
}.
\]
Then, there exists an isometric embedding of valued fields $h \colon L' \hookrightarrow M$ that induces~$f$ and~$g$ and makes the following diagram commute:
\[\xymatrix{
& L' \ar[d]^h\\
L \ar[r] \ar[ru] &M
}.\]
\end{theorem}
\begin{proof}
By~\cite[Corollary~5]{Poonenmax}, the field~$M$ may be embedded in a Mal'cev-Neumann field~$M'$ with same value group and residue field. Since~$M$ is maximally complete, the immediate extension~$M'/M$ is an isomorphism. We deduce that~$M$ is a Mal'cev-Neumann field. (Conversely, any Mal'cev-Neumann field with divisible value group and algebraically closed residue field is algebraically closed and maximally complete by~\cite[Theorem~1 and Corollary~4]{Poonenmax}.)

The result now follows from~\cite[Theorem~2]{Poonenmax}.
\end{proof}

\begin{corollary}\label{cor:GaloisPoonen}
Let~$L$ be a valued extension of~$K$ such that the residual extension $\wti L/\wti K$ has finite transcendence degree. Let~$M$ be a valued extension of~$K$ that is  algebraically closed and maximally complete. Let $a,b \colon L \hookrightarrow M$ be two $K$-embeddings of~$L$ into~$M$. Then there exists an isometric $K$-automorphism~$\sigma$ of~$M$ such that $b = \sigma \circ a$. 
\end{corollary}
\begin{proof}
By~\cite[\S 14, $n^\circ$~4, Corollaire~2]{Bourbakialg47}, there exists an automorphism~$s$ of the residue field~$\wti M$ such that $\wti b = \wti s \circ \wti a$. By Theorem~\ref{thm:GaloisPoonen}, there exists an isometric endomorphism~$\sigma$ of~$M$ that induces~$s$ on~$\wti M$ , the identity on~$|M^\ast|$ and makes the following diagram commute:
\[\xymatrix{
& M \ar[d]^\sigma\\
L \ar[r]^b \ar[ru]^a &M
}.\]
In particular, the extension defined by~$\sigma$ is immediate. As~$M$ is maximally complete, $\sigma$~is an isomorphism.

\end{proof}


\begin{remark}
The result does not hold anymore if~$M$ is not assumed to be maximally complete as proved in~\cite{MatignonReversat} answering a question from~\cite{DworkRobbaGalois}.
\end{remark}

\begin{corollary}\label{cor:Galoisdisc}
Let~$M$ be an algebraically closed and maximally complete valued extension of~$K$. Let~$x$ be a point of~$X$. The Galois group $\mathrm{Gal}^{c}(M/K)$ acts transitively on the set of $M$-rational points of~$\pi^{-1}_{M/K}(x)$ and on the set of connected components of $\pi^{-1}_{M/K}(x) \setminus \{x_{M}\}$.
\end{corollary}
\begin{proof}
We may assume that~$K$ is algebraically closed. Let~$t$ and~$t'$ be two $M$-rational points of~$\pi^{-1}_{M/K}(x)$. They correspond to two embeddings of valued fields $\Hs(x) \hookrightarrow M$. The residue field~$\wti{\Hs(x)}$ has transcendence degree at most~1 over~$\wti K$, hence, by Corollary~\ref{cor:GaloisPoonen}, the embeddings are conjugated. Stated otherwise, there exists $\sigma \in \mathrm{Gal}^c(M/K)$ that sends~$t$ to~$t'$. This proves the first result.

Let us now consider two connected components~$D$ and~$D'$ of $\pi^{-1}_{M/K}(x) \setminus \{x_{M}\}$. By Theorem~\ref{thm:structurefibre}, they are discs. Choose two $M$-rational points~$t$ and~$t'$ in~$D$ and~$D'$ respectively. We have just proven that there exists $\sigma \in \mathrm{Gal}^c(M/K)$ such that $\sigma(t)=t'$.

By Lemma~\ref{lem:Galoisuniversal}, since the point~$x$ is fixed by~$\sigma$, the point~$x_{M}$ is fixed too. We deduce that the connected component~$D$ of $\pi^{-1}_{M/K}(x) \setminus \{x_{M}\}$ is sent into a connected subset of $\pi^{-1}_{M/K}(x)$ that contains~$t'$ but not~$x_{M}$, hence into~$D'$. Similarly, we prove that $\sigma^{-1}(D') \subseteq D$. It follows that~$\sigma$ induces an isomorphism between~$D$ and~$D'$.
\end{proof}

The rest of the section is dedicated to the proof of Theorem~\ref{thm:structurefibre}. We advise the reader who is mainly interested in the applications to differential equations to skip this part and carry on with \S\ref{sec:radius}. 

\medbreak

Let~$L^a$ be the completion of an algebraic closure of~$L$. The fibre $\pi^{-1}_{L^a/L}(x_{L})$ is finite. Since~$x$ is universal, it may be canonically lifted to~$L^a$ and the fibre is actually reduced to the point~$x_{L^a}$. We deduce that, for every connected component~$C$ of $\pi_{L/K^a}^{-1}(x)\setminus\{x_{L}\}$, the base change~$C \ho_{L} L^a$ is a union of connected components of $\pi_{L^a/K}^{-1}(x)\setminus\{x_{L^a}\}$. Recall that the projection map $X_{K^a} \to X$ is the quotient by the action of $\mathrm{Gal}(\bar{K}/K)$. In particular, it is open. Thus, in order to prove that every connected component of~$\pi_{L/K^a}^{-1}(x)\setminus\{x_{L}\}$ is a virtual open disc with boundary~$\{x_{L}\}$ that is open in~$X_{L}$, it is enough to show that every connected component of~$\pi_{L^a/K^a}^{-1}(x)\setminus\{x_{L^a}\}$ is an open disc with boundary point $\{x_{L^a}\}$ that is open in~$X_{L^a}$. From now on, we will assume that~$L$ is algebraically closed.

The proof of the theorem will be split in several steps. We first consider the case of points in the affine line.

\begin{lemma}\label{lem:structurefibredisc123}
Theorem~\ref{thm:structurefibre} holds if~$x$ belongs to the affine line.
\end{lemma}
\begin{proof}
If the point~$x$ has type~1, there is only one point above it and the result is obvious.

Assume that the point~$x$ has type~2 or~3. Then it is the unique point of the Shilov boundary of some closed disc~$D_{K^a}^+(c,R)$, with~$c\in K^a$ and~$R>0$. With the notations above, we have $x = x_{c,R}$.

Here, all the computations can be made explicitly and we check that the canonical lifting~$x_{L}$ of the point~$x$ is the point~$x_{c,R}$ in~$\A^{1,\an}_{L}$, \textit{i.e.} the unique point of the Shilov boundary of the disc~$D_{L}^+(c,R)$.

The point~$x_{L}$ has type~2 or~3, hence its complement in~$D_{L}^+(c,R)$ is a disjoint union of open discs of radius~$R$. Let~$D$ be such a disc. Assume that it is not contained in the fibre~$\pi_{L/K^a}^{-1}(x)$. Then it meets the preimage of a disc of the form $D' = D_{K^a}^-(d,R)$ with~$d \in K^a$, that is to say the disc $D_{L}^-(d,R)$. Two open discs inside an affine line that meet and have the same radius are equal, hence $D = D_{L}^-(d,R)$ and $D \cap \pi_{L/K^a}^{-1}(x) = \emptyset$. We conclude that $\pi_{L/K^a}^{-1}(x)\setminus\{x_{L}\}$ is the disjoint union of the open discs of radius~$R$ that are contained in it.

Let us finally assume that~$x$ has type~4. There exists a sequence of $K^a$-rational points $(c_{n})_{n\ge 0}$ and a sequence of positive real numbers $(R_{n})_{n\ge 0}$ such that the sequence of discs $(D_{K^a}^+(c_{n},R_{n}))_{n\ge 0}$ is decreasing with intersection~$\{x\}$. We deduce that 
\[\pi_{L/K^a}^{-1}(x) = \bigcap_{n\ge 0} D_{L}^+(c_{n},R_{n}).\]
There are two cases. Let us first assume that $\bigcap_{n\ge 0} D_{L}^+(c_{n},R_{n})$ contains no $L$-rational point. Then $\pi_{L/K^a}^{-1}(x)$ is a singleton~$\{x_{L}\}$ and~$x_{L}$ is again a point of type~4. 

Now, assume that $\bigcap_{n\ge 0} D_{L}^+(c_{n},R_{n})$ contains an $L$-rational point~$c$. Then $\pi_{L/K^a}^{-1}(x)$ is the disc~$D_{L}^+(c,R)$, with $R = \inf_{n\ge 0}(R_{n}) >0$. In this case, the point~$x_{L}$ is the unique point~$x_{c,R}$ of the Shilov boundary of~$D_{L}^+(c,R)$.

In both cases, the result of Theorem~\ref{thm:structurefibre} is obvious.
\end{proof} 

We will now turn to the case of points of arbitrary curves. Let us begin with the first part of Theorem~\ref{thm:structurefibre}. The following classical result, whose proof is based on Krasner's lemma (see~\cite[0.21]{polytopes} for instance), will be useful.

\begin{lemma}\label{lem:algebrisation}
There exists an affinoid neighbourhood~$V$ of~$x$ in~$X$ and a smooth affine algebraic curve~$\Xs$ over~$K$ such that~$V$ identifies to an affinoid domain of~$\Xs^\an$. \hfill \qed
\end{lemma}

\begin{lemma}\label{lem:typecurves}
If~$x$ has type~$i\in\{1,2\}$, then~$x_{L}$ has type~$i$. If~$x$ has type~$j\in\{3,4\}$, then~$x_{L}$ has type~$j$ or~2.
\end{lemma}
\begin{proof}
%

We have seen that the result holds for points in the affine line and we will reduce to this case. By Lemma~\ref{lem:algebrisation}, we may assume that~$X$ is the analytification of an algebraic curve. Then, there exists an affinoid domain~$X'$ of~$X$ containing~$x$ and a finite morphism~$\varphi$ from~$X'$ to an affinoid domain~$Y'$ of the affine line. Let us consider the commutative diagram
\[\xymatrix{
X_{L}' \ar[r] \ar[d]^{\varphi_{L}} \ar[r]^{\pi_{L/K^a}} & X' \ar[d]^\varphi\\
Y_{L}' \ar[r]^{\pi_{L/K^a}} & Y'\\
}.\]
By Lemma~\ref{lem:universalmorphism}, we have $\varphi_{L}(x_{L}) = \varphi(x)_{L}$. The result follows since the finite morphisms~$\varphi$ and~$\varphi_{L}$ preserve types.
\end{proof}

The following lemma will be useful here and later.

\begin{lemma}\label{lem:degree}
Let $\varphi \colon Y \to Z$ be a finite morphism between quasi-smooth $K$-analytic curves. Let~$y$ be a point of~$Y$ such that the local ring~$\Os_{Y,y}$ is a field. Then the local ring~$\Os_{Z,\varphi(y)}$ is a field too and we have 
\[[\Os_{Y,y} \colon\Os_{Z,\varphi(y)}] = [\Hs(y) \colon \Hs(\varphi(y))].\] 
\end{lemma}
\begin{proof}
Since~$\Os_{Y,y}$ is a field, the point~$y$ is not rigid, hence the point~$\varphi(y)$ is not rigid either, and~$\Os_{Z,\varphi(y)}$ is a field too.

By~\cite[Theorem~2.3.3]{bleu} or~\cite[Th\'eor\`eme~4.2]{etudelocale}, the fields~$\Os_{Y,y}$ and~$\Os_{Z,\varphi(y)}$ are Henselian. Since~$K$ has characteristic~0, the extension $\Os_{Y,y}/\Os_{Z,\varphi(y)}$ is separable and the equality follows from the beginning of the proof of~\cite[Proposition~2.4.1]{bleu}.
\end{proof}

\begin{remark}
When the extension $\Os_{Y,y}/\Os_{Z,\varphi(y)}$ is not separable, the equality $[\Os_{Y,y} \colon\Os_{Z,\varphi(y)}] = [\Hs(y) \colon \Hs(\varphi(y))]$ may fail. We refer the reader to~\cite[5.3.4.2]{families} for a counterexample due to M.~Temkin. He shows that if~$K$ is a non-trivially valued and non-algebraically closed field of characteristic~$p$, then there exists a point $z\in \E{1}{K}$ such that $\Hs(z) = K^a$. Let us consider the endomorphism~$\varphi$ of~$\E{1}{K}$ defined by $T \mapsto T^p$. The point~$z$ has precisely one preimage~$y$ and we have $[\Os_{y} \colon \Os_{z}] = p$ whereas $[\Hs(y) \colon \Hs(z)] = 1$.
\end{remark}

\begin{proposition}\label{prop:opendisc}
Every connected component of $\pi^{-1}_{L/K^a}(x)\setminus\{x_{L}\}$ is an open disc with boundary~$\{x_{L}\}$. In particular, $\pi^{-1}_{L/K^a}(x)$ is connected.
\end{proposition}
\begin{proof}
If~$x$ has type~1 or~4, then, by~\cite[Th\'eor\`eme~4.3.5]{RSSen}, it admits a neighbourhood that is isomorphic to a disc and we may apply Lemma~\ref{lem:structurefibredisc123}. If~$x$ has type~3, then, by~\cite[Th\'eor\`eme~4.3.5]{RSSen}, it admits a neighbourhood that is isomorphic to an annulus and we may use Lemma~\ref{lem:structurefibredisc123} again.


Let us now assume that~$x$ has type~2. Since the result only depends on the field~$\Hs(x)$, by Lemma~\ref{lem:algebrisation}, we may assume that~$X$ is the analytification of a smooth algebraic curve over~$K^a$. In particular, $X$~has no boundary. 

Moreover, we deduce that the residue field~$\wti{\Hs(x)}$ is the function field of an irreductible and reduced algebraic curve~$\Cs_{x}$ over the algebraically closed field~$\wti{K^a}$. By generic smoothness, there exists a non-empty Zariski open subset of~$\Cs_{x}$ that admits an \'etale map to the affine line, hence there exists $\ti \alpha \in \wti{\Hs(x)}$ such that the extension $\wti{\Hs(x)}/\wti{K^a}(\ti \alpha)$ is finite and separable. Let us choose $\alpha\in \Os_{x}$ that lifts~$\ti\alpha \in \wti{\Hs(x)}$. It induces a morphism~$\varphi$ from a neighbourhood of~$x$ in~$X$ to $Y = \E{1}{K^a}$. By~\cite[Proposition~3.1.4]{bleu}, up to restricting~$X$ and~$Y$, we may assume that~$\varphi$ is finite and that $\varphi^{-1}(\varphi(x)) = \{x\}$. Since~$X$ and~$Y$ are quasi-smooth, the morphism~$\varphi$ is also flat. By construction, the extension $\wti{\Hs(x)}/\wti{\Hs(\varphi(x))}$ is finite and separable. Denote its degree by~$d$. Set $y=\varphi(x)$.

Since~$K^a$ is algebraically closed and~$x$ is of type~2, we have $|\Hs(x)^\ast| = |\Hs(y)^\ast| = |{K^a}^\ast|$. By~\cite[Corollary~6.3.6]{stablemodification} or~\cite[Th\'eor\`eme~4.3.14]{RSSen}, the field~$\Hs(y)$ is stable, hence
\[ [\Hs(x) \colon \Hs(y)] = [\wti{\Hs(x)} \colon \wti{\Hs(y)}] \, [|\Hs(x)^\ast| \colon |\Hs(y)^\ast|] = [\wti{\Hs(x)} \colon \wti{\Hs(y)}] = d.\]
Moreover, the local rings~$\Os_{x}$ and~$\Os_{y}$ are fields and, by Lemma~\ref{lem:degree}, we also have $[\Os_{x} \colon \Os_{y}] = d$. In particular, the morphism~$\varphi$ has degree~$d$.

Let~$\ti \beta \in \wti{\Hs(x)}$ such that $\wti{\Hs(y)}[\ti\beta] = \wti{\Hs(x)}$. Choose~$\beta\in \Os_{x}$ that lifts~$\ti\beta$. It is easy to check that we also have $\Os_{y}[\beta]=\Os_{x}$. Let $p(T) \in \Os_{y}[T]$ be the monic minimal polynomial of~$\beta$ over~$\Os_{y}$. It has degree~$d$ and its coefficients lie in~$\Hs(y)^\circ$. In particular, the image of~$p'(T)$ by the isomorphism $\Hs(y)[T]/(p(T)) \xrightarrow[]{\sim} \Hs(x)$ lies in~$\Hs(x)^\circ$. Since~$\ti\beta$ is separable, the reduction of~$p'(T)$ is non-zero, hence the image of~$p'(T)$ in~$\Hs(x)$ has absolute value~1.



Let~$Y'=\Mc(\Bs)$ be an affinoid domain of~$Y$ containing~$y$ such that the coefficients of~$p(T)$ belong to~$\Bs$. Set $\As = \Bs[T]/(p(T))$ and $X'=\Mc(\As)$. Let $\psi\colon X' \to Y'$ be the natural morphism. The preimage of~$y$ by~$\psi$ coincides with the spectrum of $\Hs(y)[T]/(p(T)) \simeq \Hs(x)$, hence it contains a single point and the completed residue field of the latter is isomorphic to~$\Hs(x)$. We still denote this point by~$x$.

Let us now consider the commutative diagram
\[\xymatrix{
X_{L}' \ar[r] \ar[d]^{\psi_{L}} \ar[r]^{\pi_{L/K^a}} & X' \ar[d]^\psi\\
Y_{L}' \ar[r]^{\pi_{L/K^a}} & Y'\\
}.\]
Since~$y$ is a point of type~2 of the affine line, the residue field~$\wti{\Hs(y)}$ is purely of transcendence degree~1, hence isomorphic to~$\wti{K^a}(u)$, where~$u$ is an indeterminate. By assumption, the reduction~$\wti{p}(T)$ of~$p(T)$ is irreducible over $\wti{\Hs(y)}$. Since~$\wti{K^a}$ is algebraically closed, by~\cite[Proposition~4.3.9]{EGAIV2}, the ring $\wti{K^a}(u) [T]/(\wti p(T))\otimes_{\wti{K^a}} \wti{L} = \wti L(u) [T]/(\wti p(T))$ is a domain, hence the polynomial~$\wti p(T)$ is still irreducible over $\wti L(u) \simeq \wti{\Hs(y_{L})}$. In particular, $p(T)$ is irreducible over~$\Hs(y_{L})$ and there is only one point in~$X'_{L}$ above~$y_{L}$, which can only be~$x_{L}$, by Lemma~\ref{lem:universalmorphism}. 



We have just shown that $\varphi_{L}^{-1}(y_{L})=\{x_{L}\}$, hence it is enough to prove that the preimage of every connected component of $\pi^{-1}_{L/K^a}(y)\setminus\{y_{L}\}$ is an open disc whose boundary contains~$x_{L}$.

By Lemma~\ref{lem:structurefibredisc123}, the connected components of~$\pi^{-1}_{L/K^a}(y)\setminus\{y_{L}\}$ are open discs. Let~$D$ be one of them and choose a coordinate~$S$ on it. Let~$z \in D(L)$. Let~$p_{z}(T)$ be the image of~$p(T)$ in~$\Hs(z)[T] \simeq L[T]$. There are exactly~$d$ points $z_{1},\dotsc,z_{d}$ in~$X'_{L}$ above~$z$, which correspond to the zeroes of~$p_{z}(T)$ in~$L$. Since every~$z_{i}$ lies over~$x$, the image of~$p_{z}'(T)$ in~$\Hs(z_{i})$ has absolute value~1. We deduce that the reduction $\wti{p_{z}}(T)$ is separable over~$\wti{\Hs(z)}$, hence the reductions $\ti z_{1},\dotsc, \ti z_{d} \in \wti L$ are distinct. Moreover, for every $i \in \cn{1}{d}$, we have $\wti{p_{z}}'(\ti z_{i}) \ne 0$.



We may consider~$p(T)$ as a polynomial with coefficients in~$\Os(D)$, and even $\Os(D)^\circ = L^\circ\llbracket S\rrbracket$. By Henselianity, the zeroes $\ti z_{1},\dotsc, \ti z_{d} \in \wti L$ lift to~$d$ elements of $L^\circ\llbracket S\rrbracket$, hence giving~$d$ sections of~$\varphi_{L}$ over~$D$. Of course, they correspond to the zeroes of~$p(T)$ in~$\Os(D)$.

For every point~$z'$ of~$D$ of type~2, 3 or~4, the ring~$\Os(D)$ embeds into~$\Hs(z')$, hence the~$d$ sections are distinct above~$z'$. For every point~$z'$ of~$D$ of type~1, by redoing the previous argument with~$z'$ instead of~$z$, we prove that the sections are also distinct above~$z'$. We deduce that the sections are disjoint everywhere, hence~$\varphi_{L}^{-1}(D)$ is a disjoint union of~$d$ connected components $C_{1},\dotsc,C_{d}$. 

Let $i \in \cn{1}{d}$. The map~$\varphi_{L}$ induces a map $C_{i} \to D$ that is a finite morphism of degree~1, hence an isomorphism. This proves the first part of the result. Moreover, $X_{L}'$~is compact and~$C_{i}$ is not, hence the boundary~$B_{i}$ of~$C_{i}$ is non-empty. The boundary of~$D$ in~$Y'_{L}$ is~$\{y_{L}\}$. For every neighbourhood~$U$ of~$y_{L}$ in~$Y'_{L}$, $B_{i}$ belongs to~$\psi^{-1}(U)$. We deduce that $\psi(B_{i})=\{y_{L}\}$, hence $B_{i} = \{x_{L}\}$, which concludes the proof.



\end{proof}

We still have to prove that the connected components of $\pi^{-1}_{L/K^a}(x)\setminus\{x_{L}\}$, which are isomorphic to open discs, are open in~$X_{L}$. This is a general fact that is related to the analytic version of Zariski's Main Theorem (see~\cite[Proposition~3.1.4]{bleu} for morphisms with no boundary, which is enough for our needs, or~\cite[Th\'eor\`eme~3.2]{variationdimension} in general).

\begin{lemma}\label{lem:ccopen}
Let $\varphi \colon Y \to Z$ be a morphism between quasi-smooth $K$-analytic curves. Assume that~$Y$ has no boundary and that~$\varphi$ is injective. Then $\varphi$ is open.
\end{lemma}
\begin{proof}
Let~$y\in Y$. By~\cite[Proposition~3.1.4]{bleu}, there exist affinoid neighbourhoods~$V$ and~$W$ of~$y$ and~$\varphi(y)$ such that~$\varphi$ induces a finite morphism $\psi \colon V \to W$. Since~$\psi$ is a finite map between quasi-smooth curves, it is flat, hence open by~\cite[Proposition~3.2.7]{bleu}.

%
%
%
\end{proof}


%

\begin{remark}\label{rem:previousproof}
In a previous version of this article, a different strategy was used to prove Proposition~\ref{prop:opendisc}. It involved reduction techniques that were very close to the ones used in the proof of Theorem~\ref{thm:universal} in~\cite{Angie}. In the case of points of type~2, we were able to prove that there exists an affinoid domain~$V$ of~$X_{L}$ containing~$\pi^{-1}_{L/K^a}(x)$ such that the point~$x_{L}$ is sent by the reduction map $V \to \wti{V}$ to a generic point whereas every other point of~$\pi^{-1}_{L/K^a}(x)$ is sent to a smooth point. We could then conclude by a result of S.~Bosch that ensures that the preimage of a smooth point by the reduction map is an open disc (see \cite[Satz~6.3]{Bosch}). Unfortunately, the last result being only available in the strictly affinoid case and over non-trivially valued fields, we had to distinguish several cases, which made the paper longer and more technical. The current proof of Proposition~\ref{prop:opendisc} was suggested by a referee. Let us finally mention that Lemma~\ref{lem:ccopen} can also be proven by use of those reduction techniques, using the simple fact that the preimage of a closed point is open.
\end{remark}

\begin{remark}
Theorem~\ref{thm:structurefibre} actually holds regardless of the characteristic of the field. As it is written, our proof works only when~$K$ has characteristic~0 (contrary to the one described in Remark~\ref{rem:previousproof}) because of Lemma~\ref{lem:degree}. If~$K$ is not trivially valued, this can be easily fixed by moving slightly the element~$\beta \in \Os_{x}$ that we choose in the proof of Proposition~\ref{prop:opendisc} (and using Krasner's lemma) in order to assume that $\Os_{x}/\Os_{y}$ is separable. 
 
If~$K$ is trivially valued, a direct proof is possible. Indeed, if~$K$ is algebraically closed, which we can assume, any type~2 point~$x$ appears inside the analytification~$X^\an$ of a smooth connected projective algebraic curve~$X$. By~\cite[1.4.2]{rouge}, the point~$x$ is then the only point of type~2 in~$X^\an$ and its complement is a disjoint union of open discs. Since analytification commutes with extension of scalars, the result easily follows, by arguments that are close to those we used in the proof of Lemma~\ref{lem:structurefibredisc123}.
\end{remark}

\begin{remark}
When working on this paper, Antoine Ducros told us that he also had a proof of Theorem~\ref{thm:structurefibre} (in any characteristic) which is based on computations of \'etale cohomology groups. At that time, is was not available in written form, but the reader may now find it in~\cite[\S5.3]{RSSen}.


\end{remark}

\subsection{Radius of convergence}\label{sec:radius}

For the rest of the article, we assume that~$X$ is endowed with a weak triangulation~$S$.

\begin{definition}\label{def:DxS}
Let~$x \in X$. Let~$L$ be a complete valued extension of~$K$ such that $X_{L}$ contains an $L$-rational point~$t_x$ over~$x$. We denote by~$D(t_x,S_{L})$ the biggest open disc centred at~$t_x$ that is contained in $X_{L}\setminus S_{L}$, \textit{i.e.} the connected component of $X_{L}\setminus \Gamma_{S_{L}}$ that contains~$t_x$. 
\end{definition}

\begin{remark}\label{rem:defsimple}
Assume that~$x\notin \Gamma_{S}$. In this case, the connected component~$C$ of $X\setminus\Gamma_{S}$ that contains~$x$ is a virtual disc and $D(t_x,S_{L})$ is the connected component of~$C_{L}$ that contains~$t_x$. 

Assume that $x\in \Gamma_{S}$. In this case, $D(t_x,S_{L})$ is the biggest open disc centred in~$t_x$ that is contained in~$\pi_{L}^{-1}(x)$. 

In particular, the definition of the disc $D(t_x,S_{L})$ depends only on the skeleton~$\Gamma_{S}$ and not on the weak triangulation~$S$ itself. 
\end{remark}

The following lemma is an easy consequence of the definitions.

\begin{lemma}\label{lem:basechangedisc}
Let $M/L$ be an extension of complete valued fields over~$K$. Let~$t_x$ be an $L$-rational point of~$X_{L}$. It naturally gives rise to an $M$-rational point~$t_{x,M}$ of~$X_{M}$. Then we have a natural isomorphism
\[D(t_x,S_{L})\ho_{L} M \xrightarrow[]{\sim} D(t_{x,M},S_{M}).\]
\end{lemma}

The next one follows from Corollary~\ref{cor:GaloisPoonen}.

\begin{lemma}\label{lem:isodisc}
Let $M$ be an algebraically closed and maximally complete valued extension of~$K$. Let~$t_x$ and~$t_x'$ be two points in~$X_{M}$ that project onto the same point on~$X$. Then, there exists $\sigma \in \mathrm{Gal}^c(M/K)$ that sends~$t_x$ to~$t_x'$ and induces an isomorphism $D(t_x,S_{M}) \xrightarrow[]{\sim} D(t_x',S_{M})$.
\end{lemma}
\begin{proof}
By Corollary~\ref{cor:GaloisPoonen}, there exists $\sigma \in \mathrm{Gal}^c(M/K)$ that sends~$t_x$ to~$t_x'$. By Lemma~\ref{lem:Sinvar}, the skeleton~$S_{M}$ is invariant under $\mathrm{Gal}^c(M/K)$. Hence, the isomorphism~$\psi_{\sigma}$ of~$X_{M}$ induced by~$\sigma$ sends the disc $D(t_x,S_{M})$ to a disc that contains~$t_x'$ and does not meet~$S_{M}$. We deduce that $\psi_{\sigma}(D(t_x,S_{M})) \subset D(t_x',S_{M})$. Using the same argument with~$\sigma^{-1}$, one shows the reverse inclusion.
\end{proof}

In the introduction, we explained that the radius of convergence was to appear as the radius of some disc. Unfortunately, the radius of a disc is not invariant by isomorphism. This leads us to define the radius of convergence as a relative radius inside a fixed bigger disc. The lemma that follows will help to show that it is well-defined.

\begin{lemma}\label{lem:isometry}
Let $R_{1},R_{2} >0$. Up to a translation of the coordinate~$t$, any isomorphism $\alpha \colon D_{K}^-(0,R_{1}) \xrightarrow[]{\sim} D_{K}^-(0,R_{2})$ is given by a power series of the form 
\[f(t) = \sum_{i\ge 1} a_{i} t^i \in K\llbracket t \rrbracket,\] 
with $|a_{1}| =  R_{2}/R_{1}$ and, for every $i \ge 2$, $|a_{i}| \le R_{2}/R_{1}$. In particular, it multiplies distances by the constant factor $R_{2}/R_{1}$: for every complete valued extension~$L$ of~$K$ and every $x,y$ in $D_{K}^-(0,R_{1})(L)$, we have
\[|\alpha(x)-\alpha(y)| = \frac{R_{2}}{R_{1}}\, |x-y|.\]
As a consequence, such an isomorphism may only exist when $R_{2}/R_{1} \in |K^*|$.
\end{lemma}


We may now adapt the usual definition of radius of convergence (see~\cite[\S4.4]{finiteness}, as well as~\cite[Notation~11.3.1 and Definition~11.9.1]{pde}).

\begin{definition}\label{def:radius}
Let~$\Fs$ be a locally free $\Os_{X}$-module of finite type with an integrable connection~$\nabla$. Let~$x$ be a point in~$X$ and~$L$ be a complete valued extension of~$K$ such that $X_{L}$ contains an $L$-rational point~$t_x$ over~$x$. Let us consider the pull-back~$(\Fs^x,\nabla^x)$ of~$(\Fs,\nabla)$ on $D(t_x,S_{L}) \simeq D_{L}^-(0,R)$. Let $r=\mathrm{rk}(\Fs^{x})$. For $i\in \cn{1}{r}$, we denote by~$\Rc'_{S,i}(x,(\Fs,\nabla))$ the radius of the biggest open subdisc of~$D_{L}^-(0,R)$ centred at~$0$ on which the connection $(\Fs^x,\nabla^x)$ admits at least~$r-i+1$ horizontal sections that are linearly independent  over~$L$. Let us define the $i^\textrm{th}$~radius of convergence of $(\Fs,\nabla)$ at~$x$ by $\Rc_{S,i}(x,(\Fs,\nabla)) = \Rc'_{S,i}(x,(\Fs,\nabla))/R$ and the multiradius of convergence of $(\Fs,\nabla)$ at~$x$ by
\[\bRc_{S}(x,(\Fs,\nabla)) = (\Rc_{S,1}(x,(\Fs,\nabla)),\dots,\Rc_{S,r}(x,(\Fs,\nabla))) \in (0,1]^r.\]
We will frequently suppress~$\nabla$ from the notation when it is clear from the context.
\end{definition}

\begin{remark}\label{rem:RS1}
With the previous notations, one may also consider the radius of the biggest open subdisc of~$D_{L}^-(0,R)$ centred at~$0$ on which the connection $(\Fs^x,\nabla^x)$ is trivial, as in~\cite[Definition~3.1.7]{ContinuityCurves}, for instance. This way, one recovers the radius~$\Rc_{S,1}$. 
\end{remark}

\begin{remark}\label{rem:radiiGammaS}
By Remark~\ref{rem:defsimple}, the radii depend only on the skeleton~$\Gamma_{S}$ and not on the weak triangulation~$S$ itself.
\end{remark}

Definition~\ref{def:radius} is independent of the choices made and invariant by extensions of the ground field~$K$, thanks to the preceding lemmas (first prove the independence of the isomorphism $D(t_x,S_{L}) \simeq D_{L}^-(0,R)$ and in particular of~$R$, then the invariance by base-change for rational points and finally reduce to the case where~$L$ is algebraically closed and maximally complete). We state the following for future reference.

\begin{lemma}\label{lem:basechange}
Let~$L$ be a complete valued extension of~$K$. For any $x\in X_{L}$, we have
\[\bRc_{S_{L}}(x,\pi_{L}^*(\Fs,\nabla)) = \bRc_{S}(\pi_{L}(x),(\Fs,\nabla)).\]
\end{lemma}

Let us now explain how the function behaves with respect to changing triangulations. Let~$S'$ be a weak triangulation of~$X$ such that~$\Gamma_{S'}$ contains~$\Gamma_{S}$. Let $x \in X$. Let~$L$ be a complete valued extension of~$K$ such that $X_{L}$ contains an $L$-rational point~$t_x$ over~$x$. By~Remark~\ref{rem:defsimple}, the disc $D(t_x,S'_{L})$ is included in $D(t_x,S_{L}) \simeq D_{L}^-(0,R)$. Let~$R'$ be its radius as a sub-disc of $D_{L}^-(0,R)$ and set $\rho_{S',S}(x) = R'/R \in (0,1]$. Remark that $\rho_{S',S}$ is constant and equal to~1 on~$\Gamma_{S}$. It is now easy to check that, for any $i\in\cn{1}{\mathrm{rk}(\Fs_{x})}$, we have
\begin{equation}\label{eq:rhoS'S}
\Rc_{S',i}(x,(\Fs,\nabla)) = \min \left(\frac{\Rc_{S,i}(x,(\Fs,\nabla))}{\rho_{S',S}(x)},1\right).
\end{equation}

\subsection{Comparison with other definitions}\label{sec:comparison}

We compare the radius of convergence we have introduced in Definition~\ref{def:radius} to other radii that appear in the literature.

\subsubsection{F.~Baldassarri's definition using semistable models}\label{sec:Baldassarri}

The first definition of radius of convergence on a curve has been given by F.~Baldassarri in~\cite{ContinuityCurves}. It was our main source of inspiration and our definition is very close to his.

Assume that the absolute value of~$K$ is non-trivial and that the curve~$X$ is strictly $K$-affinoid. In this case, it is known that there exits a finite separable extension~$L/K$ such that the curve~$X_{L}$ admits a semistable formal model~$\Xk$ over~$L^\circ$.

There actually exists a strong relation between the semistable models of~$X_{L}$ and its triangulations (see~\cite[\S6.4]{RSSen}, for a detailed account). Indeed, for any generic point~$\ti\xi$ of the special fibre~$\Xk_{s}$ of~$\Xk$, let~$\xi$ be the unique point of the generic fibre $\Xk_{\eta} = X_{L}$ whose reduction is equal to~$\ti\xi$. Gathering all points~$\xi$, we construct a finite set~$S(\Xk)$ that is a triangulation of~$X_{L}$. Let us remark that our notation unfortunately disagrees with F.~Baldassarri's: in~\cite{ContinuityCurves}, $S(\Xk)$~denotes the skeleton (\textit{i.e.} $\Gamma_{S(\Xk)}$ using our notations) and not the triangulation.

Let~$x \in X_{L}(L)$. In~\cite{ContinuityCurves}, F.~Baldassarri considers the biggest disc that does not meet the skeleton~$\Gamma_{S(\Xk)}$ of~$\Xk$ (see \cite[Definition~1.6.6]{ContinuityCurves}). It is nothing but our disc~$D(x,S(\Xk))$. Since every point of~$S(\Xk)$ has type~2, this disc is actually isomorphic to the open unit disc~$D_{L}^-(0,1)$, and F.~Baldassarri defines the radius of convergence~$\Rc_{\Xk}(x,\pi^*_{L}(\Fs,\nabla))$ of~$\pi_{L}^*(\Fs,\nabla)$ at~$x$ as the radius~$r$ of the biggest open subdisc of $D_{L}^-(0,1)$ centred at~0 on which~$(\Fs,\nabla)$ is trivial (see~\cite[Definition~3.1.8]{ContinuityCurves}). This is compatible with our definition that uses a relative radius (see Lemma~\ref{lem:isometry} and Remark~\ref{rem:RS1}). Finally, we have proved that, for any $x\in X_{L}(L)$, we have
\[\Rc_{\Xk}(x,\pi_{L}^*(\Fs,\nabla)) = \Rc_{S_{L},1}(x,\pi^*_{L}(\Fs,\nabla)).\]

F.~Baldassarri extends his definition to other points of the curve by extending the scalars so as to make them rational (see~\cite[Definition~3.1.11]{ContinuityCurves}). One may check that, for any complete valued extension~$M$ of~$L$, $\Xk\ho_{L^\circ} M^\circ$ is as semistable model of~$X_{M}$ and that $S(\Xk\ho_{L^\circ} M^\circ) = S(\Xk)_{M}$. Hence our definition coincides with F.~Baldassarri's everywhere. 

Let us also point out that, conversely, for any triangulation~$S$ of~$X$ that only contains points of type~2, there exits a finite separable extension~$L/K$ and a semistable formal model~$\Xk$ of~$X_{L}$ such that $S(\Xk) = S_{L}$. Hence, under the hypotheses of this section and if we restrict to triangulations that contains only points of type~2, our definition is essentially equivalent to F.~Baldassarri's. (Remark that, in the non-strict case, triangulations must be allowed to contain points of type~3. We also allow this in general, as A.~Ducros does.)

\medskip

Finally, let us mention that F.~Baldassarri actually considers a slightly more general situation: $X = \bar{X}\setminus\{z_{1},\ldots,z_{r}\}$, where~$\bar{X}$ is a compact curve as above and $z_{1},\dots,z_{r}$ are $K$-rational points. In this case, he constructs the skeleton of~$X$ by branching on the skeleton of~$\bar{X}$ a half-line~$\ell_{i}$ that goes in the direction of~$z_{i}$, for each~$i$. The definition of radius of convergence may then be adapted.

Let us mention that this more general situation is already covered in our setting, since we did not require the curves to be compact. To find the same skeleton, it is enough to begin with the triangulation of~$\bar{X}$ and add, for each~$i$, a sequence of points that lie on~$\ell_{i}$ and tend to~$z_{i}$.

\subsubsection{The definition for analytic domains of the affine line}\label{sec:affineline}

Assume that~$X$ is an analytic domain of the affine line~$\A^{1,\an}_{K}$. The choice of a coordinate~$t$ on~$\A^{1,\an}_{K}$ provides a global coordinate on~$X$ and it seems natural to use it in order to measure the radii of convergence. This normalisation has been used by F.~Baldassarri and L.~Di Vizio in~\cite{ContinuityBDV} (for the first radius) and by the second author in~\cite{finiteness}. We will call ``embedded'' the radii we define in this setting. 

From now on, we will assume that~$X$ is not the affine line. Let us first give a definition of radii that does not refer to any triangulation.

\begin{definition}\label{def:multiradiusemb}
Let~$x$ be a point of~$X$ and~$L$ be a complete valued extension of~$K$ such that $X_{L}$ contains an $L$-rational point~$t_x$ over~$x$. Let $D(t_x,X_{L})$ be the biggest open disc centred at~$t_x$ that is contained in~$X_{L}$ (which exists since $X \ne \E{1}{K}$).

Let us consider the pull-back~$(\Fs^x,\nabla^x)$ of~$(\Fs,\nabla)$ on $D(t_x,X_{L})$. Let $r=\mathrm{rk}(\Fs^{x})$. For $i\in \cn{1}{r}$, we denote by $\Rc_{i}^\mathrm{emb}(x,(\Fs,\nabla))$ the radius of the biggest open subdisc of $D(t_x,X_{L})$ centred at~$t_x$, measured using the coordinate~$t$ on~$\A^{1,\an}_{L}$, on which the connection $(\Fs^x,\nabla^x)$ admits at least~$r-i+1$ horizontal sections that are linearly independent  over~$L$.
\end{definition}

As before, one checks that the definition of~$\Rc_{i}^\mathrm{emb}(x,(\Fs,\nabla))$ only depends on the point~$x$ and not on~$L$ or~$t_x$. This radius is denoted by~$\Rc_{i}^M(x)$ in~\cite[\S4.4]{finiteness}. There, the second author works over an affinoid domain~$V$ of the affine line. Over~$V$, the sheaf~$\Fs$ is free and, if~$\Os(V)$ is endowed with the usual derivation $d = \mathrm{d}/\mathrm{d}t$, the pair~$(\Fs,\nabla)$ corresponds to a differential module~$(M,D)$ over~$(\Os(V),d)$. This is where the~$M$ in the notation comes from.

\medbreak

Although possibly superfluous in this context, it is also possible to state a definition that depends on a weak triangulation~$S$ of~$X$.

\begin{definition}\label{def:multiradiusembS}
Let~$x$ be a point of~$X$ and~$L$ be a complete valued extension of~$K$ such that $X_{L}$ contains an $L$-rational point~$t_x$ over~$x$. As in Definition~\ref{def:DxS}, consider $D(t_x,S_{L})$, the biggest open disc centred at~$t_x$ that is contained in $X_{L}\setminus S_{L}$. We denote by~$\rho_{S}(x)$ its radius, measured using the coordinate~$t$ on~$\A^{1,\an}_{L}$. 

Let us consider the pull-back~$(\Fs^x,\nabla^x)$ of~$(\Fs,\nabla)$ on $D(t_x,S_{L})$. Let $r=\mathrm{rk}(\Fs_{x})$. For $i\in \cn{1}{r}$, we denote by $\Rc^\mathrm{emb}_{S,i}(x,(\Fs,\nabla))$ the radius of the biggest open subdisc of $D(t_x,S_{L})$ centred at~$t_x$, measured using the coordinate~$t$ on~$\A^{1,\an}_{L}$, on which the connection $(\Fs^x,\nabla^x)$ admits at least~$r-i+1$ horizontal sections that are linearly independent  over~$L$.
\end{definition}

Once again, the definitions of~$\rho_{S}(x)$ and~$\Rc^\mathrm{emb}_{S,i}(x,(\Fs,\nabla))$ are independent of the choices of~$L$ and~$t_x$. This radius is denoted by~$\Rc^M_{S,i}(x)$ in~\cite[\S8]{finiteness}.

\medskip

The radii we have just defined may easily be linked to the one we introduced in Definition~\ref{def:radius}. The simplest case is the second: for any $i\in \cn{1}{\mathrm{rk}(\Fs_{x})}$, we have
\begin{equation}\label{eq:radiusembS}
\Rc_{S,i}(x,(\Fs,\nabla)) = \frac{\Rc^\textrm{emb}_{S,i}(x,(\Fs,\nabla))}{\rho_{S}(x)}.
\end{equation}

Since~$X$ is not the affine line, we have the following result, whose proof we leave to the reader.

\begin{lemma}\label{lem:minimalskeleton}
The set of skeletons of the weak triangulations of~$X$ admits a smallest element~$\Gamma_{0}(X)$. It coincides with the analytic skeleton of~$X$, \textit{i.e.} the set of points that have no neighbourhood isomorphic to a virtual open disc. \hfill \qed
\end{lemma}

Let~$S_{0}$ be a weak triangulation of~$X$ whose skeleton is~$\Gamma_{0}$. By Remark~\ref{rem:radiiGammaS}, the radii computed with respect to a weak triangulation only depend on the skeleton of the latter. In particular, the radii computed with respect to~$S_{0}$ are well-defined. It is now easy to check that, for every $i\in \cn{1}{\mathrm{rk}(\Fs_{x})}$, we have
\begin{equation}\label{eq:radiusembS0}
\Rc_{S_{0},i}(x,(\Fs,\nabla)) = \frac{\Rc^\textrm{emb}_{S_{0},i}(x,(\Fs,\nabla))}{\rho_{S_{0}}(x)} =  \frac{\Rc_{i}^\textrm{emb}(x,(\Fs,\nabla))}{\rho_{S_{0}}(x)}.
\end{equation}

\section{The result}\label{sec:result}

We have just defined the radius of convergence of~$(\Fs,\nabla)$ at any point of the curve~$X$. We now investigate its properties.

\subsection{Statement}

Let us state precisely the result we are interested in. We will first define the notion of locally finite subgraph of~$X$.

From the existence of weak triangulations, one deduces that every point of~$X$ has a neighbourhood that is uniquely arcwise connected. In particular, the curve~$X$ may be covered by uniquely arcwise connected analytic domains. On such a subset, it makes sense to speak of the segment~$[x,y]$ joining two given points~$x$ and~$y$, hence of convex subsets (see also~\cite[\S2.5]{BR}).

\begin{definition}
A subset~$\Gamma$ of~$X$ is said to be a finite (resp. locally finite) subgraph of~$X$ if there exists a finite (resp. locally finite)  family~$\Vs$ of affinoid domains of~$X$ that covers~$\Gamma$ and such that, for every element~$V$ of~$\Vs$, we have
\begin{enumerate}
\item $V$ is uniquely arcwise connected;
\item $\Gamma \cap V$ is the convex hull of a finite number of points.
\end{enumerate}
\end{definition}

We now want to define a notion of $\log$-linearity. To do so, we first need to explain how to measure distances. 

\begin{definition}\label{defi:modulus}
Let~$C$ be a closed virtual annulus over~$K$. Its preimage over~$K^a$ is a finite union of closed annuli. If $C_{K^a}^+(c,R_{1},R_{2})$ is on of them, we set
\[\textrm{Mod}(C) = \frac{R_{2}}{R_{1}}.\]
It is independent of the choices.
\end{definition}

\begin{notation}
Let~$I$ be a closed interval inside a virtual disc or annulus and assume that~$I$ contains only points of type~2 or~3. Then~$I$ is the skeleton of a virtual closed subannulus~$I^\sharp$ and we set 
\[\ell(I) = \log(\textrm{Mod}(I^\sharp)).\]
\end{notation}

\begin{remark}
Pushing these ideas further, one can show that it is possible to define a canonical length for any closed interval inside a curve that contains only points of type~2 or~3 (see~\cite[Proposition~4.5.7]{RSSen}). The definition may actually be extended to the whole curve see~\cite[Corollaire~4.5.8]{RSSen}).
\end{remark}

\begin{definition}\label{defi:loglinear}
Let~$X_{[2,3]}$ be the set of points of~$X$ that are of type~2 or~3. Let~$J$ be an open interval inside~$X_{[2,3]}$ and identify it with a real interval. A map $f \colon J \to \R_{+}^\ast$ is said to be $\log$-linear if there exists~$\gamma\in \R$ such that, for every $a < b \in J$, we have
\[\log(f(b)) - \log(f(a)) = \gamma\, \ell([a,b]).\] 

Let~$\Gamma$ be a locally finite subgraph of~$X$. A map $f \colon \Gamma \to \R_{+}^\ast$ is said to be piecewise $\log$-linear if~$\Gamma$ may be covered by a locally finite family~$\Js$ of closed intervals such that, for every $J\in\Js$, the restriction of~$f$ to~$\mathring{J}$ is $\log$-linear.
\end{definition}

We may now state the theorem we want to prove.

\begin{theorem}\label{thm:continuousandfinite}
The map~$\bRc_{S}$ satisfies the following properties:
\begin{enumerate}[\it i)]
\item it is continuous;
\item its restriction to any locally finite graph~$\Gamma$ is piecewise $\log$-linear;
\item on any interval~$J$, the $\log$-slopes of its restriction are rational numbers of the form~$m/j$ with $m\in\Z$ and $j\in\cn{1}{r}$, where~$r$ is the rank of~$\Fs$ around~$J$; 
\item there exists a locally finite subgraph~$\Gamma$ of~$X$ and a continuous retraction $r\colon X \to \Gamma$ such that the map~$\bRc_{S}$ factorises by~$r$.
\end{enumerate}
\end{theorem}

The continuity of the radius of convergence~$\Rc_{S,1}$ has been proven by F.~Baldassarri and L.~Di Vizio in the case of affinoid domains of the affine line (see~\cite{ContinuityBDV}) and by F.~Baldassarri in general (see~\cite{ContinuityCurves}). His setting is actually slightly less general than ours, but his result extends easily. 

For the multiradius of convergence on affinoid domains of the affine line, the result is due to the second author (see~\cite[Theorem~4.6.3]{finiteness} for the case of the minimal weak triangulation and~\cite[\S8]{finiteness} in general).

\begin{theorem}[(A.~Pulita)]\label{thm:Andrea}
Theorem~\ref{thm:continuousandfinite} holds when~$X$ is an affinoid domain of the affine line.
\end{theorem}

Let us write down a corollary of Theorem~\ref{thm:Andrea} about open discs that will be useful later.

\begin{corollary}\label{cor:continuousopendisc}
Assume that~$X$ is an open disc endowed with the empty triangulation. Let~$x$ be a point of~$X$ and let~$[x,y]$ be a segment with initial point~$x$. The restriction of the map~$\bRc_{S}$ to the segment~$[x,y]$ is continuous at the point~$x$ and $\log$-linear in the neighbourhood of~$x$ with slopes of the form~$m/j$ with $m\in\Z$ and $j\in\cn{1}{\rk(\Fs)}$.
\end{corollary}
\begin{proof}
We will identify~$X$ with the disc $D^-(0,R)$ for some~$R>0$. Let $i\in\cn{1}{\rk(\Fs)}$. Pick $r\in (0,1)$ such that $[x,y] \subset D^-(0,rR)$.

Let us first assume that $\Rc_{\emptyset,i}(\wc,\Fs)$ is smaller than or equal to~$r$ on~$[x,y]$. Let us endow~$X$ with the triangulation $S = \{x_{0,rR}\}$. By Formula~(\ref{eq:rhoS'S}), for every $z\in [x,y]$, we have 
\[\Rc_{S,i}(z,\Fs) = r\, \Rc_{\emptyset,i}(z,\Fs).\] 
Moreover, for every~$z\in D^+(0,rR)$, we have $\Rc_{S,i}(z,\Fs) = r\, \Rc_{S,i}(z,\Fs_{|D^+(0,rR)})$.
The result now follows from Theorem~\ref{thm:Andrea}.

We may now assume that there exists $z\in [x,y]$ and $r'\in (r,1)$ such that $\Rc_{\emptyset,i}(z,\Fs) = r'$. This means that the module~$\Fs$ has a trivial submodule of corank~$i-1$ on~$D^-(0,r'R)$, hence $\Rc_{\emptyset,i}(\wc,\Fs)=r'$ on~$D^-(0,r'R)$ and the result follows trivially. 
\end{proof}

\begin{remark}
Let us mention that if one is only interested in the first radius of convergence~$\Rc_{S,1}$ on boundary-free curves (or with overconvergent connections), it is possible to get the result of Theorem~\ref{thm:continuousandfinite} in a much shorter way \textit{via} potential theory (see~\cite{potential}).
\end{remark}

\begin{remark}
The result of Theorem~\ref{thm:continuousandfinite} can be strengthened. For instance, one can require that the graph~$\Gamma$ contains no point of type~4. This is equivalent to the fact that the radii are constant in the neighbourhood of every point of type~4, a result that is due to K.~Kedlaya (see~\cite[Lemma~4.5.14]{Kedlayalocalglobal}).
\end{remark}

\begin{remark}
Let~$J$ be an interval inside~$X$. One could expect that the restriction of the map~$\bRc_{S}$ to~$J$ is $\log$-linear around every point of type~3 that does not belong to~$S$. We know how to prove it for~$\Rc_{S,1}$ or when no radius is solvable but the general case seems trickier. It is related to the general question of super-harmonicity for the partial heights of the convergence Newton polygon (\textit{i.e.} the sum of the logarithms of the first $j$ slopes for varying~$j$), which is open (see~\cite[Remark~2.4.10]{finiteness4} for the description of an hypothetical situation where it would fail).


Nevertheless, we are able to prove that if~$\Gamma_{S}$ has no end-points of type~3, then the smallest graph~$\Gamma$ that contains~$\Gamma_{S}$ and satisfies the properties of Theorem~\ref{thm:continuousandfinite} (for the multi-radius~$\bRc_{S}$ and not some individual~$\Rc_{S,i}$) exists and has no end-points of type~3 either. We refer to~\cite{finiteness3} for more on this issue.
\end{remark}

Let us now turn to the proof of Theorem~\ref{thm:continuousandfinite}. By Lemma~\ref{lem:basechange}, it is enough to prove the result after a scalar extension. \textit{From now on, we will assume that~$K$ is algebraically closed.}

\subsection{A geometric result}\label{sec:geometry}

The overall strategy of our proof is the following. By definition of a weak triangulation, the set $X\setminus S$ is a union of open discs and annuli and, for those, we may use the results of the second author (see Theorem~\ref{thm:Andrea}). We still need to investigate what happens around the points of the triangulation~$S$. To carry out this task, we will write the curve~$X$, locally around those points, as a finite \'etale cover of a subset of the affine line, consider the push-forward of $(\Fs,\nabla)$, which is well understood thanks to Theorem~\ref{thm:Andrea} again, and relate its radii of convergence to the original radii.

We will need to find \'etale morphisms from open subsets of~$X$ to the affine line that satisfy nice properties. The main result we use has been adapted from the proof of~\cite[Th\'eor\`eme~4.4.15]{RSSen}. For the convenience of the reader, we have decided to sketch the proof of the whole result. 

In what follows, we use A.~Ducros's notion of ``branch'' (see~\cite[\S1.7]{RSSen}), which roughly corresponds to that of a direction out of a point. More precisely, for every open subtree~$V$ of~$X$ containing a point~$x$, there is a bijection between the set of branches out of~$x$ and the set of connected components of~$V\setminus\{x\}$. Such a connected component is called a section of the corresponding branch. Every branch admits a section that is isomorphic to an open annulus.

There are well-defined notions of direct and inverse images of branches that correspond to the intuitive ones. Let $\varphi\colon X \to Y$ be a morphism of curves and let~$x\in X$ be a point such that~$\varphi^{-1}(\varphi(x))$ is finite. Then the image of a branch out of~$x$ is a branch out of~$\varphi(x)$ and the preimage of a branch out of~$\varphi(x)$ is a union of branches out of~$x$.

Let~$x$ be a point of~$X$ of type~2 and consider the complete valued field~$\Hs(x)$ associated to it. Since~$\wti{K}$ is algebraically closed, the residue field~$\wti{\Hs(x)}$ is the function field of a unique projective smooth connected curve~$\Cs$ over~$\wti{K}$, called the residual curve (see~\cite[3.3.5.2]{RSSen}). If $x$ lies in the interior of~$X$, the closed points of~$\Cs$ correspond bijectively to the branches out of the point~$x$ (see~\cite[4.2.11.1]{RSSen}).

We say that a property holds for almost every element of a set~$E$ if it holds for all but finitely many of them.

\begin{theorem}\label{thm:bonvois}
Let~$x$ be a point of~$X$ of type~2. Let~$b_{1},\dotsc,b_{t},c$ be distinct branches out of~$x$. Let~$N$ be a positive integer.  There exists an affinoid neighbourhood~$Z$ of~$x$ in~$X$, a quasi-smooth affinoid curve~$Y$, an affinoid domain~$W$ of~$\P^{1,\textrm{an}}_{K}$ and a finite \'etale map $\psi \colon Y \to W$ such that
\begin{enumerate}
\item $Z$ is isomorphic to an affinoid domain of~$Y$ and~$x$ lies in the interior of~$Y$;
\item\label{i:deg} the degree of~$\psi$ is prime to~$N$;
\item\label{i:pointx} $\psi^{-1}(\psi(x))=\{x\}$;
\item\label{i:compx} almost every connected component of $Y\setminus\{x\}$ is an open unit disc with boundary~$\{x\}$;
\item\label{i:compfx} almost every connected component of $W\setminus\{\psi(x)\}$ is an open unit disc with boundary~$\{\psi(x)\}$;
\item\label{i:compiso} for almost every connected component~$V$ of $Y\setminus\{x\}$, the induced morphism $V \to \psi(V)$ is an isomorphism;
\item\label{i:isobi} for every $i \in \cn{1}{t}$, the morphism~$\psi$ induces an isomorphism between a section of~$b_{i}$ and a section of~$\psi(b_{i})$ and we have $\psi^{-1}(\psi(b_{i})) \subseteq Z$;
\item\label{i:branchec} $\psi^{-1}(\psi(c)) = \{c\}$. 
\end{enumerate}

\end{theorem}
\begin{proof}
By Lemma~\ref{lem:algebrisation}, there exists an affinoid neighbourhood~$Z$ of~$x$ in~$X$ and a smooth affine algebraic curve~$\Xs$ over~$K$ such that~$V$ identifies to an affinoid domain of~$\Xs^\an$. We may assume that~$\Xs$ is projective and connected. Let~$\Cs$ be the residual curve at the point~$x$. Let~$g$ be a rational function on~$\Cs$ that induces a generically \'etale morphism $\Cs \to \P^1_{\wti{K}}$. Let~$f$ be a rational function on~$\Xs$ such that $|f(x)|=1$ and $\wti{f} = g$. Let $f^\an \colon \Xs^\an \to \P^{1,\textrm{an}}_{K}$ be the associated morphism.


Let us first remark that, for almost every connected component~$V$ of $\Xs^\an\setminus\{x\}$, $V$ meets a unique branch out of~$x$ and $f^\an(V)$ is a connected component of $\P^{1,\an}_{K} \setminus\{f^\an(x)\}$.

Since the map induced by~$g$ is generically \'etale, for almost every connected component~$V$ of $\Xs^\an\setminus\{x\}$, it is unramified at the closed point of~$\Cs$ corresponding to the branch associated to~$V$. From this we deduce that the morphism $V \to f^\an(V)$ induced by~$f^\an$ has degree~1 (see~\cite[Th\'eor\`eme~4.3.13]{RSSen}), hence is an isomorphism. 

Finally, choose an affinoid neighbourhood~$W$ of~$f^\an(x)$ in~$\P^{1,\textrm{an}}_{K}$ such that the different points of~$(f^\an)^{-1}(f^\an(x))$ belong to different connected components of~$(f^\an)^{-1}(W)$. Let~$Y$ be the connected component containing~$x$ and $\psi \colon Y \to W$ be the induced morphism. Properties~(\ref{i:pointx}) and~(\ref{i:compiso}) are satisfied by construction. Since~$\psi(x)$ is a point of type~2, property~(\ref{i:compfx}) is clear. Property~(\ref{i:compx}) follows.

For any~$i\in\cn{1}{t}$, let~$\wti{b}_{i}$ be the closed point of the residue curve~$\Cs$ associated to the branch~$b_{i}$. Let~$\wti{c}$ be the closed point associated to the branch~$c$. Finally, let $\wti{a}_{1},\dotsc,\wti{a}_{u}$ be the closed points associated to the branches of~$\Xs$ out of~$x$ that do not belong to~$X$. In the lemma below, we will prove that there exists a rational function~$g$ as above whose degree~$d$ is prime to~$N$, that has a simple zero at every~$\wti{b}_{i}$, a unique pole at~$\wti{c}$ and takes the value~1 at every~$\wti{a}_{j}$. Property~(\ref{i:branchec}) then follows from the link between branches out of a point and closed points of the residue curve. The first part of property~(\ref{i:isobi}) follows from~\cite[Th\'eor\`eme~4.3.13]{RSSen}, as above. The second follows from the fact that $g^{-1}(g(\wti{b}_{i})) = g^{-1}(0)$ contains none of the~$\wti{a}_{j}$'s. Moreover, if we assume that the set of branches~$b_{i}$ is non-empty, which we can always do, then the map~$g$ has a simple zero, which forces it to be generically \'etale.

As regards the degree of~$\psi$, note that we have $[\wti{\Hs(x)}\colon\wti{\Hs(\psi(x))}] = \deg(g) = d$. Since~$K$ is algebraically closed and~$x$ is of type~2, we have $|\Hs(x)^\ast| = |\Hs(\psi(x))^\ast| = |K^\ast|$. By~\cite[Corollary~6.3.6]{stablemodification} or~\cite[Th\'eor\`eme~4.3.14]{RSSen}, the field~$\Hs(\psi(x))$ is stable, hence
\[ [\Hs(x) \colon \Hs(\psi(x))] = [\wti{\Hs(x)} \colon \wti{\Hs(\psi(x))}] \, [|\Hs(x)^\ast| \colon |\Hs(\psi(x))^\ast|] = d.\]
By Lemma~\ref{lem:degree}, the degree of~$\psi$ at~$x$ is~$d$. Since~$\psi^{-1}(\psi(x)) = \{x\}$, the degree of~$\psi$ itself is also~$d$.


\end{proof}

\begin{lemma}
Let~$C$ be a projective smooth connected curve over a field~$k$. Let $x_{1},\dotsc,x_{t},y,z_{1},\dotsc,z_{u}$ be distinct closed points of~$C$. Let~$N,n_{1},\dotsc,n_{t}$ be positive integers. There exists a rational function~$g$ on~$C$ with degree prime to~$N$ such that
\begin{enumerate}
\item for every $i\in\cn{1}{t}$, $g$ has a zero of order~$n_{i}$ at~$x_{i}$;
\item for every $j\in\cn{1}{u}$, $g$ takes the value~1 at~$z_{j}$;
\item $g$ has a unique pole, which lies at~$y$.
\end{enumerate}
\end{lemma}
\begin{proof}
Let~$d$ be a positive integer. Consider the divisors 
\[D = (d-1)(y) - (n_{1}+1)(x_{1}) - \dotsb - (n_{t}+1)(x_{t}) - (z_{1}) - \dotsb - (z_{u})\]
and
\[D' = d(y) - n_{1}(x_{1}) - \dotsb - n_{t}(x_{t})\]
on~$C$. The cokernel of the natural injection $\Os(D)\to \Os(D')$ is a skyscraper sheaf~$\Gs$ supported on $\{x_{1},\dotsc,x_{t},y,z_{1},\dotsc,z_{u}\}$. Let~$s$ be the global section of~$\Gs$ that takes the value~1 at every point of the support.

Choose~$d$ big enough so as to have $H^1(C,\Os(D))=0$. We may also assume that~$d$ is prime to~$N$. Then we have an exact sequence
\[0 \to H^0(C,\Os(D)) \to H^0(C,\Os(D')) \to H^0(C,\Gs) \to 0.\]
Let~$s'$ be an element of $H^0(C,\Os(D'))$ that lifts~$s$. The associated rational function satisfies the required properties.
\end{proof}

%

\subsection{Proof of the finiteness property}\label{sec:prooffinite}

In this section, we will prove that the map~$\bRc_{S}(\wc,(\Fs,\nabla))$ is locally constant outside a locally finite subgraph~$\Gamma$ of~$X$.

By definition of a triangulation, $X\setminus S$ is a union of open discs and annuli, each of which may be handled by Theorem~\ref{thm:Andrea}. We still need to investigate the behaviour of the multiradius around the points of the triangulation. Let us first remark that, as far as the finiteness property is concerned, it is harmless to change triangulations.

\begin{lemma}\label{lem:change}
Let~$S$ and~$S'$ be two weak triangulations of~$X$. There exists a locally finite subgraph~$\Gamma$ of~$X$ outside which the map $\bRc_{S}(\wc,\Fs)$ is locally constant if, and only if, there exists a locally finite subgraph~$\Gamma'$ of~$X$ outside which the map $\bRc_{S'}(\wc,\Fs)$ is locally constant.
\end{lemma}
\begin{proof}
It is possible to construct a triangulation~$S''$ that contains both~$S$ and~$S'$. Hence we may assume that $S \subset S'$. In this case, Formula~(\ref{eq:rhoS'S}) shows that the property for~$S$ implies the property for~$S'$.

Let us assume that there exists a locally finite subgraph~$\Gamma'$ of~$X$ outside which the map $\bRc_{S'}(\wc,\Fs)$ is locally constant. We may assume that~$\Gamma'$ contains~$\Gamma_{S'}$. Let~$U$ be a connected component of~$X\setminus \Gamma'$. It is enough to prove that the map $\bRc_{S}(\wc,\Fs)$ is constant on~$U$. Let~$V$ be the connected component of~$X\setminus \Gamma_{S}$ that contains~$U$. Both~$U$ and~$V$ are discs and the distance function~$\rho_{S,S'}$ (see the paragraph preceding formula~(\ref{eq:rhoS'S})) is constant on~$U$. Let~$\rho$ be its value. Let $r=\mathrm{rk}(\Fs_{|U})$. For any $i\in\cn{1}{r}$ and any $x\in U$, we have
\[\Rc_{S',i}(x,\Fs) = \min \left(\frac{\Rc_{S,i}(x,\Fs)}{\rho},1\right).\]
We now have two cases. Fix $i\in\cn{1}{r}$. If there exists $x\in U$ such that $\Rc_{S,i}(x,\Fs)$ is at least~$\rho$, then $\Rc_{S,i}(\wc,\Fs)$ is constant on~$U$ (which is contained in an open disc of relative radius~$\rho$). Otherwise, the maps $\Rc_{S,i}(\wc,\Fs)$ and $\rho\, \Rc_{S',i}(\wc,\Fs)$ coincide on~$U$, hence both are constant.
\end{proof}

In our study, we will need to restrict the connection to some subspaces. Unfortunately, the multiradius of convergence may vary under this operation. In the following lemma, we gather a few easy cases where the resulting multiradius may be controlled. Recall that we denote by~$x_{c,R}$ the unique point of the Shilov boundary of~$D^+(c,R)$.

\begin{lemma}\label{lem:restrictionhorsGammaS}
Let~$x$ be a point of~$\Gamma_{S}$. Let~$C$ be an open disc or annulus inside~$X\setminus \Gamma_{S}$ such that $\bar{C}\cap \Gamma_{S} =\{x\}$. 
\begin{enumerate}[a)]
\item Assume that~$C$ is an open disc. Then, for any $y\in C$, we have 
\[\bRc_{\emptyset}(y,\Fs_{|C}) = \bRc_{S}(y,\Fs).\]
\item Assume that~$C$ is an open annulus. Identify~$C$ with an annulus $C^-(0,R_{1},R_{2})$, with coordinate~$t$, in such a way that $\lim_{R\to R_{2}} x_{0,R} = x$. Then, for any $y\in C$ and any $i \in \cn{1}{\mathrm{rk}(\Fs_{|C})}$, we have 
\[\Rc_{\emptyset,i}(y,\Fs_{|C}) = \min\left(\frac{R_{2}}{|t(y)|}\, \Rc_{S,i}(y,\Fs), 1\right).\]
Moreover, if $\Rc_{\emptyset,i}(x_{0,R},\Fs_{|C}) = 1$ for all~$R$ close enough to~$R_{2}$, then either $\Rc_{S,i}(\wc,\Fs)=1$ on~$C$ or $\Rc_{S,i}(x_{0,R},\Fs)=R/R_{2}$ for all~$R$ close enough to~$R_{2}$.
\end{enumerate}
\end{lemma}
\begin{proof}
Assume we are in case~(a). The set~$C$ is an open disc and the point~$x$ lies at its boundary. As a consequence, for any complete valued extension~$L$ of~$K$ and any $L$-rational point~$\ti y$ of~$X_{L}$, the disc $D(\ti y,S_{L})$ is equal to~$C_{L}$. In particular, the multiradius of convergence of $\Fs$ on~$C$ only depends on the restriction of $\Fs$ to~$C$ and the result follows.

Assume we are in case~(b). The connected component of~$X\setminus \Gamma_{S}$ that contains~$C$ is an open disc~$D$. We may identify it with $D^-(0,R_{2})$, with coordinate~$t$, in a way that is compatible with the identification of~$C$ and~$C^-(0,R_{1},R_{2})$. By case~(a), restricting the connection to~$D$, endowed with the empty weak triangulation, leaves the radii unchanged. 

Consider the weak triangulation $T = \{x_{0,R_{1}}\}$ of~$D$. Its skeleton is $\Gamma_{T} = \{x_{0,R} \mid R_{1}\le R < R_{2}\}$. Since $\Gamma_{C} = C \cap \Gamma_{T}$ and radii only depend on skeletons, for every $y\in C$, we have 
\[\bRc_{\emptyset}(y,\Fs_{|C}) = \bRc_{T}(y,\Fs_{|D}).\]
We may now apply Formula~(\ref{eq:rhoS'S}) to compute the right-hand side and the result follows.

Let us now prove the final statement. By case~(a), we have $\Rc_{S,i}(\wc,\Fs) = \Rc_{\emptyset,i}(\wc,\Fs_{|D})$ on~$D$. Let us assume that there exists $R'\in (R_{1},R_{2})$ such that $\Rc_{\emptyset,i}(x_{0,R},\Fs_{|C}) = 1$ for all $R\in (R',R_{2})$. This implies that $\Rc_{\emptyset,i}(x_{0,R},\Fs_{|D})  \ge R/R_{2}$ for all $R\in (R',R_{2})$. If the latter is an equality everywhere, we are done.

Otherwise, there exists~$R'' \in (R',R_{2})$ such that $\Rc_{\emptyset,i}(x_{0,R''},\Fs_{|D}) > R''/R_{2}$. We may then write this radius in the form~$r/R_{2}$, with $r\in (R'',R_{2}]$. This means that the module~$\Fs$ has a trivial submodule of corank~$i-1$ on~$D^-(0,r)$, hence $\Rc_{\emptyset,i}(\wc,\Fs_{|D})=r/R_{2}$ on~$D^-(0,r)$. If~$r=R_{2}$, we deduce that $\Rc_{\emptyset,i}(\wc,\Fs_{|D})  = 1$ on~$D$. If~$r<R_{2}$, then we have $\Rc_{\emptyset,i}(x_{0,R},\Fs_{|D})  = R/R_{2}$ for every $R\in [r,R_{2})$. Indeed, if it were not the case, then we would be able to repeat the previous argument with some $R''' \in (r,R_{2})$ and show that $\Rc_{\emptyset,i}(\wc,\Fs_{|D}) = s/R_{2}$, with $s\in (r,R_{2}]$, in the neighbourhood of~0 in~$D$. This is a contradiction.
\end{proof}

\begin{lemma}\label{lem:restrictionGammaS}
Let~$C$ be an open annulus inside~$X\setminus S$ such that $\Gamma_{C} \cap \Gamma_{S} \ne \emptyset$. Then, we have $\Gamma_{C} = C \cap \Gamma_{S}$ and, for any $y\in C$, \[\bRc_{\emptyset}(y,\Fs_{|C}) = \bRc_{S}(y,\Fs).\]
\end{lemma}
\begin{proof}
Let~$C'$ be the connected component of~$X\setminus S$ that contains~$C$. Since $\Gamma_{C} \cap \Gamma_{S} \ne \emptyset$, $C'$~is an annulus and we have $\Gamma_{C'} = C'\cap \Gamma_{S}$. An inclusion of annuli whose skeletons meet induces an inclusion between their skeletons and we deduce that $\Gamma_{C} = C \cap \Gamma_{S}$. The result is now proved as case~(a) of Lemma~\ref{lem:restrictionhorsGammaS} using Remark~\ref{rem:defsimple}.
\end{proof}

Using Theorem~\ref{thm:bonvois}, we now prove a kind of generic finiteness of the multiradius around a point of the triangulation that is of type~2.

\begin{proposition}\label{prop:locfinite}
Let~$x$ be a point of~$S$ of type~2. There exists an affinoid domain~$Y_{x}$ of~$X$ such that $Y_{x} \cap S =\{x\}$, every connected component of~$Y_{x}\setminus\{x\}$ is an open disc and, for every $i\in\cn{1}{r}$, 
the map $\Rc_{S,i}(\wc,\Fs)_{|Y_{x}}$ is locally constant outside a finite subgraph~$\Gamma_{x}$ of~$Y_{x}$ that contains~$x$.
\end{proposition}
\begin{proof}
Let us consider a finite \'etale map $\psi \colon Y_{x} \to W_{x}$ as in Theorem~\ref{thm:bonvois} (with no~$b_{i}$'s nor~$c$ and $N=1$). Remark that almost every branch out of~$x$ belongs to~$X$. It is possible to restrict~$W_{x}$ and~$Y_{x}$ by removing a finite number of connected components of~$W_{x}\setminus\{\psi(x)\}$ and~$Y_{x}\setminus\{x\}$ in order to assume that~$Y_{x}$ is an affinoid domain of~$X$, that $S\cap Y_{x} =\{x\}$ and that properties~(\ref{i:compx}), (\ref{i:compfx}) and~(\ref{i:compiso}) of Theorem~\ref{thm:bonvois} hold for every, and not only almost every, connected component that appears in their statements. Beware that~$Y_{x}$ will no longer be a neighbourhood of~$x$. 


Since~$\psi$ is finite \'etale, we may consider the push-forward $\psi_{*}(\Fs,\nabla)_{|Y_{x}}$ of $(\Fs,\nabla)_{|Y_{x}}$ to~$W_{x}$. By property~(\ref{i:compfx}), the subset $T = \{\psi(x)\}$ of~$W_{x}$ is a weak triangulation of~$W_{x}$.

Let~$V$ be a connected component of $Y_{x}\setminus\{x\}$. By property~(\ref{i:compiso}), its image~$\psi(V)$ is a connected component of $W_{x}\setminus\{\psi(x)\}$, hence an open disc. Let $V' = \psi^{-1}(\psi(V))$ and let $V_{1}=V,V_{2},\dotsc,V_{d}$ be its connected components. By properties~(\ref{i:compx}) and~(\ref{i:compiso}), for every~$i \in \cn{1}{d}$, $V_{i}$ is a disc and the induced morphism $\psi_{i} \colon V_{i} 
\to \psi(V)$ is an isomorphism. By case~(a) of Lemma~\ref{lem:restrictionhorsGammaS}, for every~$y$ in~$V'$, we have
\[\bRc_{S}(y,\Fs) =  \bRc_{\emptyset}(y,\Fs_{|V'})\] 
and, for every~$z$ in~$\psi(V)$, 
\[\bRc_{T}(z,\psi_{*}\Fs_{|Y_{x}}) =  \bRc_{\emptyset}(z,(\psi_{*}\Fs_{|Y_{x}})_{|\psi(V)}) = \bRc_{\emptyset}(z,\psi'_{*}(\Fs_{|V'})),\]
where we denote by $\psi' \colon V' \to \psi(V)$ the induced morphism.


Since~$\psi'$ is a trivial cover of degree~$d$, the situation is simple. Actually, the module $\psi'_{*}(\Fs_{|V'})$ over~$\psi(V)$ splits 
as $\bigoplus_{1\le i\le d} {\psi_{i}}_{*}(\Fs_{|V_{i}})$. By~\cite[Proposition~5.1.1]{finiteness}, we deduce that, for every $y\in V$, every component of the multiradius of 
convergence $\bRc_{S}(y,\Fs)$ appears in $\bRc_{T}(\psi(y),\psi_{*}\Fs_{|Y_{x}})$.

By Theorem~\ref{thm:Andrea}, the map $\bRc_{T}(\wc,\psi_{*}\Fs_{|Y_{x}})$ is locally constant outside a finite 
subgraph~$\Gamma_{x}$ of~$W_{x}$ that contains~$x$. This finite graph only meets a finite number of connected components $U_{1},\dotsc,U_{n}$ of~$W_{x}\setminus\{x\}$. For every~$i\in\cn{1}{n}$, $\Gamma_{x} \cap (U_{i} \cup\{\psi(x)\})$ is a finite graph containing~$x$. By property~(\ref{i:pointx}), we have $\psi^{-1}(\psi(x))=\{x\}$ and, by property~(\ref{i:compiso}), the induced morphism $\psi^{-1}(U_{i}) \to U_{i}$ is a trivial cover. We deduce that $\psi^{-1}(\Gamma_{x} \cap (U_{i} \cup\{x\}))$ is a finite graph, hence $\psi^{-1}(\Gamma_{x})$ too. We will prove that the map $\bRc_{S}(\wc,\Fs)$ is locally constant on its complement.

Let~$U$ be a connected open subset of 
$Y_{x}\setminus \psi^{-1}(\Gamma_{x})$. It is contained in a connected component~$D$ of $Y_{x} \setminus\{x\}$. By property~(\ref{i:compx}), $D$ is an open disc. It contains no point of~$S$ and the point~$x$ lies at its boundary. By Lemma~\ref{lem:restrictionhorsGammaS}, case~(a), and Corollary~\ref{cor:continuousopendisc}, the map $\bRc_{S}(\wc,\Fs)$ is continuous on~$D$, hence on~$U$.

For every point~$y\in U$, the components of the $r$-tuple $\bRc_{S}(y,\Fs)$ are equal to some of the components of the $dr$-tuple 
$\bRc_{T}(\psi(y),\psi_{*}\Fs)$. Since $\bRc_{T}(\wc,\psi_{*}\Fs)$ is constant on~$\psi(U)$, there are only finitely many 
possible such values. We deduce that the restriction of the map $\bRc_{S}(\wc,\Fs)$ to~$U$ is continuous with values in a finite set, hence constant.
\end{proof}

\begin{remark}\label{rem:ccYx}
Every connected component of~$Y_{x}\setminus\{x\}$ is a connected component of~$X\setminus S$.
\end{remark}

Let~$x$ be a point of~$S$ of type~2. The affinoid domain~$Y_{x}$ of the previous corollary contains entirely almost every branch out of 
$x$. We still need to prove the finiteness property on the remaining branches. This is the object of the following proposition (where we also handle points of type~3).

\begin{proposition}\label{prop:finitebranch}
Let~$x$ be a point of~$S$ and~$b$ be a branch out of it. There exists an open annulus~$C_{x,b}$ that is a section of~$b$ and satisfies $\bar{C}_{x,b}\cap S =\{x\}$ as well as a finite subgraph~$\Gamma_{x,b}$ of~$\bar{C}_{x,b}$ such that the map $\bRc_{\emptyset}(\wc,\Fs_{|C_{x,b}})$ is locally constant outside $\Gamma_{x,b} \cap C_{x,b}$.
\end{proposition}
\begin{proof}
By~\cite[Th\'eor\`eme~4.3.5]{RSSen}, any point of type~3 has a neighbourhood that is isomorphic to a closed annulus. Hence, around such a point, we may conclude by Theorem~\ref{thm:Andrea}.

Let us now assume that~$x$ is a point of type~2. The proof will closely follow that of Proposition~\ref{prop:locfinite}. Let us consider a finite \'etale map $\psi \colon Y \to W$ as in Theorem~\ref{thm:bonvois} with $b_{1}=b$ (with no other~$b_{i}$'s, no~$c$ and $N=1$). Let~$C_{x,b}$ be an open annulus that is a section of~$b$, that satisfies $\bar{C}_{x,b}\cap S = \{x\}$, $\psi^{-1}(\psi(C_{x,b})) \subseteq X$ and such that the induced map $\psi_{0} \colon C_{x,b} \to \psi(C_{x,b})$ is an isomorphism. We may restrict~$W$ and~$Y$ by removing a finite number of connected components of~$W\setminus\{\psi(x)\}$ and~$Y\setminus\{x\}$ in order to assume that~$Y$ is an affinoid domain of~$X$ containing~$\psi^{-1}(\psi(C_{x,b})) \cup \{x\}$ (but not necessarily a neighbourhood of~$x$ anymore). Let us consider the push-forward $\psi_{*}(\Fs,\nabla)$ of the connection $(\Fs,\nabla)$ to~$W$. We endow~$W$ with a weak triangulation~$T$ such that $T\cap \psi(\bar{C}_{x,b}) = \partial \psi(C_{x,b}) = \psi(\partial C_{x,b})$.

Arguing as in the proof of Proposition~\ref{prop:locfinite}, we show that, for every~$z$ in $\psi(C_{x,b})$, we have
\[\bRc_{T}(z,\psi_{*}\Fs) = \bRc_{\emptyset}(z,\psi'_{*}(\Fs_{|P_{x,b}})),\]
where we denote by $\psi' \colon P_{x,b} = \psi^{-1}(\psi(C_{x,b})) \to \psi(C_{x,b})$ the induced morphism.


The subset~$P_{x,b}$ has several connected components, one of which is~$C_{x,b}$. We deduce that $(\psi_{0})_{*}(\Fs_{|C_{x,b}})$ is a direct factor of $\psi'_{*}(\Fs_{|P_{x,b}})$. Since~$\psi_{0}$ is an isomorphism, for every $y\in C_{x,b}$, every 
component of the multiradius $\bRc_{\emptyset}(y,\Fs_{|C_{x,b}})$ appears in $\bRc_{T}(\psi(y),\psi_{*}\Fs)$. 
By Theorem~\ref{thm:Andrea}, the map $\bRc_{T}(\wc,\psi_{*}\Fs)$ is locally constant outside a finite subgraph~$\Gamma$ of~$W$. Using an argument of continuity as in the last paragraph of the proof of Proposition~\ref{prop:locfinite}, we deduce that the map $\bRc_{\emptyset}(\wc,\Fs_{|C_{x,b}})$ is locally constant outside~$\psi^{-1}(\Gamma)$. Remark that $\psi^{-1}(\Gamma)$ is a finite subgraph of~$\bar{C}_{x,b}$.
\end{proof}

\begin{remark}
In the situation of Proposition~\ref{prop:finitebranch}, let~$C$ be an open annulus that is a section of~$b$. The coefficients of the matrix of the connection on~$C$ converge in a neighbourhood of~$\bar{C}$ in~$X$. If~$X$ were an affinoid domain of the affine line, we would deduce that these coefficients are analytic elements on~$C$ (in the sense of~\cite[\S8.5]{pde}) and then conclude by~\cite[Corollary~4.6.5]{finiteness}. Unfortunately, in the general case, such functions do not give rise to analytic elements.
\end{remark}

We may now conclude the proof of the finiteness property. By Lemma~\ref{lem:change}, we may enlarge the weak triangulation~$S$ into a triangulation in the sense of A.~Ducros, which means that every connected component of~$X\setminus S$ is relatively compact.

Let~$S_{[2]}$ (resp.~$S_{[3]}$) be the set of points of~$S$ that are of type~2 (resp.~3). To every~$x\in S_{[2]}$, we associate an affinoid domain~$Y_{x}$ of~$X$ containing~$x$ by Proposition~\ref{prop:locfinite}. Let $b_{x,1},\dotsc,b_{x,n(x)}$ be the branches out of~$x$ that do not belong to~$Y_{x}$. For every~$x\in S_{[3]}$, we set~$Y_{x}=\{x\}$ and denote by $b_{x,1},\dotsc,b_{x,n(x)}$ the branches out of~$x$ (hence $n(x)\le 2$).

To every~$x\in S$ and every~$i\in\cn{1}{n(x)}$, we associate an open annulus~$C_{x,b_{x,i}}$ by Proposition~\ref{prop:finitebranch}. We may shrink the annuli~$C_{x,b_{x,i}}$ in order to assume that they do not overlap. We now enlarge the triangulation~$S$ of~$X$ into a triangulation~$S'$ such that every annulus~$C_{x,b_{x,i}}$ is a connected component of~$X\setminus S'$. By Lemma~\ref{lem:change} again, this does not affect the result we want to prove. 

For every~$x\in S_{[2]}$, we have constructed a finite subgraph~$\Gamma_{x}$ of~$Y_{x}$ outside which the map $\bRc_{S}(\wc,\Fs)_{|Y_{x}}$, hence also the map $\bRc_{S'}(\wc,\Fs)_{|Y_{x}}$ is locally constant. Set 
\[\Gamma_{Y} = \bigcup_{x\in S_{[2]}} \Gamma_{x}.\] 
Every point of~$X$ has a neighbourhood that meets at most one of the~$Y_{x}$'s, hence~$\Gamma_{Y}$ is locally finite.

For every~$x\in S$ and every $i\in\cn{1}{n(x)}$, we have constructed a finite subgraph~$\Gamma_{x,b_{x,i}}$ of~$\bar C_{x,b_{x,i}}$ such that the map $\bRc_{\emptyset}(\wc ,\Fs_{|C_{x,b_{x,i}}}) = \bRc_{S'}(\wc ,\Fs)_{|C_{x,b_{x,i}}}$ (see Lemma~\ref{lem:restrictionGammaS}) is locally constant outside $\Gamma_{x,b_{x,i}} \cap C_{x,b_{x,i}}$. Set 
\[\Gamma_{C} = \bigcup_{\substack{x\in S\\ 1\le i\le n(x)}} \Gamma_{x,b_{x,i}}.\] 
Every point of~$X$ has a neighbourhood that meets only a finite number of the~$C_{x,b_{x,i}}$'s (and actually at most one for points in~$X\setminus S$), hence~$\Gamma_{C}$ is locally finite.

Let~$\Es$ be the set of connected components of~$X\setminus \bigcup_{x\in S} Y_{x}$. Let $E\in \Es$. By Remark~\ref{rem:ccYx}, it is a connected component of~$X\setminus S$. Assume that~$E$ is a disc. Since~$S$ is a triangulation (and not only a weak triangulation), there exists a point~$x\in S$ that lies at the boundary of~$E$. Let~$b$ be the branch out of~$x$ that is defined by~$E$. It belongs to no~$Y_{x'}$ with~$x'\in S_{[2]}$, hence we can consider the annulus~$C_{x,b}$. By assumption, the boundary point~$z_{x,b}$ of~$C_{x,b}$ inside~$E$ belongs to~$S'$. The complement~$Z_{E}$ of~$C_{x,b}$ in~$E$ is a closed disc with boundary point~$z_{x,b}$. In particular, $\{z_{x,b}\}$ is a triangulation of~$Z_{E}$. By Theorem~\ref{thm:Andrea} there exists a finite subgraph~$\Gamma_{E}$ of~$Z_{E}$ such that the restriction of the map~$\bRc_{S'}(\wc,\Fs)$ to~$Z_{E}$ is locally constant outside $\Gamma_{E}$. 

If~$E$ is an annulus, we argue in a similar way to define a closed subannulus~$Z_{E}$ of~$E$ whose complement in the union of two~$C_{x,b}$'s and a finite subgraph~$\Gamma_{E}$ of~$Z_{E}$ such that the restriction of the map~$\bRc_{S'}(\wc,\Fs)$ to~$Z_{E}$ is locally constant outside $\Gamma_{E}$. Set
\[\Gamma_{\Es} = \bigcup_{E\in\Es} \Gamma_{E}.\] 
Every point of~$X$ has a neighbourhood that meets at most one of the~$Z_{E}$'s, hence~$\Gamma_{\Es}$ is locally finite.

On the whole, the set $\Gamma = \Gamma_{Y} \cup \Gamma_{C} \cup \Gamma_{\Es}$ is a locally finite subgraph of~$X$. Moreover, since~$X$ is covered by the union of the~$Y_{x}$'s, the~$C_{x,b}$'s and the~$Z_{E}$'s, the map $\bRc_{S'}(\wc ,\Fs)$ is locally constant on~$X\setminus\Gamma$.

\subsection{Proof of continuity and log-linearity at points of the skeleton}

We will now prove the continuity and $\log$-linearity at points of the skeleton in a given direction. Let us begin with the case of points of type~3, which is easier.

\begin{lemma}\label{lem:continuoustype3}
Let~$x$ be a point of~$\Gamma_{S}$ of type~3 and let~$[x,y]$ be a segment with initial point~$x$. For every $i\in\cn{1}{\mathrm{rk}(\Fs_{x})}$, the restriction of the map $\Rc_{S,i}(\wc,\Fs)$ to the segment~$[x,y]$ is continuous at the point~$x$ and $\log$-linear in the neighbourhood of~$x$ with slopes of the form~$m/j$ with $m\in\Z$ and $j\in\cn{1}{\rk(\Fs_{x})}$.
\end{lemma}
\begin{proof}
By~\cite[Th\'eor\`eme~4.3.5]{RSSen}, the point~$x$ has a neighbourhood that is isomorphic to an annulus. If $\Gamma_{S} \cap [x,y] \ne \emptyset$, then there exists~$z\in (x,y]$ and a closed annulus~$C$ in~$X$ such that $[x,z] = \Gamma_{S} \cap C$ is the skeleton of~$C$. Since the radii only depend on the skeleton, we have $\Rc_{S,i}(\wc,\Fs)_{|C} = \Rc_{\{x,z\},i}(\wc,\Fs_{|C})$ and the results we want now follow from Theorem~\ref{thm:Andrea}.
 
If $\Gamma_{S} \cap [x,y] = \emptyset$, then there exists a closed disc~$D$ in~$X$ with boundary point~$x$ that contains~$y$ and such that $D\cap S = \{x\}$. We may then conclude as before.
\end{proof}

We will now turn to the case of points of type~2. Let us first state a result that relates the multiradius of convergence after push-forward by an \'etale map to the original multiradius of convergence. Recall that~$K$ is assumed to be algebraically closed.

\begin{lemma}
Let~$Z$ be a quasi-smooth $K$-analytic curve. Let $\psi \colon X \to Z$ be a finite morphism. Let $x\in X$ be a point of type~2 or~3. Assume that $d = [\Hs(x) \colon \Hs(\psi(x))]$ is prime to~$p$.\footnote{If~$\ti K$ has characteristic~0, then $p=1$ and this condition is always satisfied.} Let~$L$ be an algebraically closed complete valued extension of~$\Hs(x)$. Then every connected component of $\pi_{L}^{-1}(\psi(x))\setminus\{\psi(x)_{L}\}$ is a disc and the morphism~$\psi_{L}$ induces a trivial cover of degree~$d$ over it.
\end{lemma}
\begin{proof}
By Theorem~\ref{thm:structurefibre} (i), the point~$x_{L}$ has type~2 or~3. Moreover, by definition of~$x_{L}$ (see Definition~\ref{defi:universal}), the norm associated to~$x_{L}$ is induced by the tensor norm on $\Hs(x)\ho_{K} L$. Hence 
\[\sqrt{|\Hs(x_{L})^*|} = \sqrt{\|\Hs(x)\ho_{K} L \setminus\{0\}\|} = \sqrt{|L^*|},\] 
since~$L$ contains~$\Hs(x)$. We deduce that the point~$x_{L}$ has type~2. As a consequence, the point~$\psi_{L}(x_{L})$, where $\psi_{L} \colon X_{L} \to Z_{L}$ is the morphism induced by~$\psi$, also has type~2.

Recall also that, by Lemma~\ref{lem:universalmorphism}, we have $\psi_{L}(x_{L}) =\psi(x)_{L}$ and that, by Theorem~\ref{thm:structurefibre}, every connected component of $\pi_{L}^{-1}(\psi(x))\setminus\{\psi(x)_{L}\}$ is a disc.

Let~$\Cs$ and~$\Ds$ be the residual curves at~$x_{L}$ and~$\psi(x)_{L}$ respectively. Using the stability of the field~$\Hs(\psi(x)_{L})$ and arguing as in the proof of Theorem~\ref{thm:bonvois}, we show that the morphism $\ti{\psi}_{L} \colon \Cs \to \Ds$ induced by~$\psi$ has degree~$d$. Since~$d$ is prime to~$p$ and $\ti{L}$ is algebraically closed, $\ti{\psi}_{L}$ is generically \'etale, hence unramified at almost every closed point.  By~\cite[Th\'eor\`eme~4.3.13]{RSSen}, we deduce that~$\psi_{L}$ has degree~1 on almost every branch out of~$x_{L}$ and the result follows for almost every connected component of $\pi_{L}^{-1}(\psi(x))\setminus\{\psi(x)_{L}\}$ (see the proof of Theorem~\ref{thm:bonvois} for more details). 

It remains to show that the result holds for every connected component of $\pi_{L}^{-1}(\psi(x))\setminus\{\psi(x)_{L}\}$. Let~$C$ be one of them. If~$L$ is maximally complete, then by Corollary~\ref{cor:GaloisPoonen}, there exists $\sigma\in \mathrm{Gal}^c(L/K)$ that maps~$C$ onto a connected component over which~$\psi_{L}$ induces a trivial cover of degree~$d$ and the result follows.

Otherwise, we embed~$L$ into a field~$M$ that is algebraically closed and maximally complete. Let~$e$ be the number of connected components of~$\psi_{L}^{-1}(C)$. It is enough to prove that~$e=d$. Since~$L$ is algebraically closed, those connected components are geometrically connected (see~\cite[Th\'eor\`eme~7.11]{excellence}) and $\pi^{-1}_{M/L}(\psi_{L}^{-1}(C)) = \psi_{M}^{-1}(\pi^{-1}_{M/L}(C))$ still has~$e$ connected components.

Finally, remark that $\pi^{-1}_{M/L}(C)$ is a connected component of $\pi_{M}^{-1}(\psi(x))\setminus\{\psi(x)_{M}\}$, hence $\psi_{M}$ induces a trivial cover of degree~$d$ over it, by the previous arguments. Hence~$e=d$.
\end{proof}

\begin{lemma}\label{lem:radiusetale}
Let~$Z$ be a quasi-smooth $K$-analytic curve endowed with a weak triangulation~$T$. Let $\psi \colon X \to Z$ be a finite \'etale morphism. Let $x\in \Gamma_{S} \cap \psi^{-1}(\Gamma_{T})$. Assume that the degree $d = [\Hs(x)\colon\Hs(\psi(x))]$ is prime to~$p$. Then, for every $i\in \cn{1}{\mathrm{rk}(\Fs_{x})}$ and $j\in \cn{1}{d}$, we have 
\[\Rc_{T,d(i-1)+j}(\psi(x),\psi_{*}(\Fs,\nabla)) = \Rc_{S,i}(x,(\Fs,\nabla)).\]
\end{lemma}
\begin{proof}
Since~$x\in \Gamma_{S}$ and $\psi(x)\in \Gamma_{T}$, the radii of convergence at those points only depend on the restrictions of the connections to those points (see Remark~\ref{rem:defsimple}). Let~$L$ be an algebraically closed complete valued extension of~$\Hs(x)$. By the previous lemma, every connected component of $\pi_{L}^{-1}(\psi(x))\setminus\{\psi(x)_{L}\}$ is a disc over which the morphism~$\psi_{L}$ induces a trivial cover of degree~$d$. 

Let~$W$ be such a connected component and~$\ti y$ be an $L$-rational point of it. Let us denote by $\psi_{L}^{-1}(\ti y) = \{\ti{x}_{1},\dots,\ti{x}_{d}\}$. Arguing as in the proof of Proposition~\ref{prop:locfinite}, we show that $\bRc_{T_{L}}(\ti{y},{\psi_{L}}_{*}\pi_{L}^*\Fs)$ is obtained by concatenating and reordering the tuples $\bRc_{S_{L}}(\ti{x}_{1},\pi_{L}^*\Fs), \dots, \bRc_{S_{L}}(\ti{x}_{d},\pi_{L}^*\Fs)$.

We have $\pi_{L}(\ti y) = \psi(x)$ and, for any $i \in \cn{1}{d}$, $\pi_{L}(\ti{x}_{i}) = x$. From Lemma~\ref{lem:basechange}, we deduce that $\bRc_{T_{L}}(\ti{y},{\psi_{L}}_{*}\pi_{L}^*\Fs) = \bRc_{T}(\psi(x),\psi_{*}\Fs)$ and, for any \mbox{$i \in \cn{1}{d}$}, $\bRc_{S_{L}}(\ti{x}_{i},\pi_{L}^*\Fs) = \bRc_{S_{L}}(x,\Fs)$. The result follows.
\end{proof}

\begin{proposition}\label{prop:continuoustype2}
Let~$x$ be a point of~$\Gamma_{S}$ of type~2 and let~$[x,y]$ be a segment with initial point~$x$. For every $i\in\cn{1}{\mathrm{rk}(\Fs_{x})}$, the restriction of the map $\Rc_{S,i}(\wc,\Fs)$ to the segment~$[x,y]$ is continuous at the point~$x$ and $\log$-linear in the neighbourhood of~$x$ with slopes of the form~$m/j$ with $m\in\Z$ and $j\in\cn{1}{\mathrm{rk}(\Fs_{x})}$.
\end{proposition}
\begin{proof}
Up to shrinking~$[x,y]$, we may assume that the open interval~$(x,y)$ is the skeleton~$\Gamma_{C}$ of an open annulus~$C$. Let~$c$ be the branch out of~$x$ associated to~$C$.

Let us consider a finite \'etale map $\psi \colon Y \to W$ that satisfies the conclusions of Theorem~\ref{thm:bonvois} with~$c=c$ and~$N=p$ (and no~$b_{i}$'s). In particular, we have~$\psi^{-1}(\psi(c)) = \{c\}$ and the degree~$d$ of~$\psi$ is prime to~$p$. Replacing~$C$ by a sub-annulus, we may assume that $C \subseteq Z$, that $S\cap C = \emptyset$ and that, for every~$z\in \Gamma_{C}$, $\psi^{-1}(\psi(z)) = \{z\}$, hence $[\Hs(z)\colon\Hs(\psi(z))] = d$, by Lemma~\ref{lem:degree}. 

Since $\psi^{-1}(\psi(C)) = C$ and $\psi^{-1}(\psi(x))=\{x\}$, we may restrict~$W$ and~$Y$ by removing a finite number of connected components of~$W\setminus\{\psi(x)\}$ and~$Y\setminus\{x\}$ in order to assume that~$Y$ is an affinoid domain of~$X$ containing~$C\cup\{x\}$ (but not necessarily a neighbourhood of~$x$ anymore). Let us endow~$W$ with a triangulation~$T$ whose skeleton contains~$\psi(\Gamma_C)$ and~$\psi(x)$.

%

By Lemma~\ref{lem:radiusetale}, for any $i\in \cn{1}{\mathrm{rk}(\Fs_{x})}$, we have 
\[\Rc_{S,i}(x,\Fs) = \Rc_{T,di}(\psi(x),\psi_{*}\Fs).\]

We will now distinguish two cases. Let us first assume that $\Gamma_{C}\cap \Gamma_{S} \ne \emptyset$. By Lemma~\ref{lem:restrictionGammaS} and Lemma~\ref{lem:radiusetale}, for any $z\in \Gamma_{C}$ and any $i\in \cn{1}{\mathrm{rk}(\Fs_{x})}$, we have 
\[\Rc_{S,i}(z,\Fs) = \Rc_{T,di}(\psi(z),\psi_{*}\Fs).\]
By Theorem~\ref{thm:Andrea}, the map $\Rc_{T,di}(\wc,\psi_{*}\Fs)$ is continuous on~$W$ and piecewise $\log$-linear on every segment inside it with slopes as requested. The result follows.

Let us now assume that $\Gamma_{C}\cap \Gamma_{S} = \emptyset$. By Lemma~\ref{lem:radiusetale} and case~(b) of Lemma~\ref{lem:restrictionhorsGammaS}, whose notations we borrow, for any $i\in \cn{1}{\mathrm{rk}(\Fs_{x})}$ and any $R\in (R_{1},R_{2})$, we have 
\[\Rc_{T,di}(\psi(x_{0,R}),\psi_{*}\Fs) = \Rc_{\emptyset,i}(x_{0,R},\Fs_{|C}) = \min\left(\frac{R_{2}}{R}\, \Rc_{S,i}(x_{0,R},\Fs),1\right).\]
If $\Rc_{T,di}(\psi(x_{0,R}),\psi_{*}\Fs) < 1$ for all~$R$ close enough to~$R_{2}$, then we have
\[\Rc_{T,di}(\psi(x_{0,R}),\psi_{*}\Fs) = \frac{R_{2}}{R}\, \Rc_{S,i}(x_{0,R},\Fs)\]
and we conclude as before. Otherwise, the set of~$R$'s such that $\Rc_{T,di}(\psi(x_{0,R}),\psi_{*}\Fs) = 1$ contains~$R_{2}$ in its closure. There exists $R'\in (R_{1},R_{2})$ such that $\Rc_{T,di}(\wc,\psi_{*}\Fs)$ is $\log$-linear on $[\psi(x_{0,R'}),\psi(x)]$, which implies that we actually have $\Rc_{T,di}(\psi(x_{0,R}),\psi_{*}\Fs) = 1$ for all~$R$ close enough to~$R_{2}$. The result then follows from the last statement of Lemma~\ref{lem:restrictionhorsGammaS} using the fact that
\[\Rc_{S,i}(x,\Fs) = \Rc_{T,di}(\psi(x),\psi_{*}\Fs) = 1.\]
\end{proof}

\subsection{Conclusion of the proof}

In \S\ref{sec:prooffinite}, we proved that there exists a locally finite subgraph~$\Gamma$ of~$X$ outside which the map $\bRc_{S}$ is locally constant. As a consequence, to prove that $\bRc_{S}$ is continuous, it is now enough to prove that it is continuous at the initial point~$x$ of every closed interval~$[x,y]$ inside~$X$. If~$x$ belongs to~$\Gamma_{S}$, we did it in Lemma~\ref{lem:continuoustype3} and Proposition~\ref{prop:continuoustype2}. Otherwise, $x$ belongs to an open disc and we may use Corollary~\ref{cor:continuousopendisc}. This proves property~(i) of Theorem~\ref{thm:continuousandfinite}.

Properties~(ii) and~(iii), which deals with piecewise $\log$-linearity and the form of the slopes, are dealt with in the same way.

Finally, remark that it is possible to enlarge~$\Gamma$ into a locally finite subgraph~$\Gamma'$ of~$X$ such that $X\setminus\Gamma'$ is a disjoint union of relatively compact open discs. In this case, by Remark~\ref{rem:retraction}, there exists a natural continuous retraction $X \to \Gamma'$. The map $\bRc_{S}$ factorises by it and Theorem~\ref{thm:continuousandfinite} is now proved.

\bigbreak

\noindent \textbf{Acknowledgements:} We thank 
A.~Ducros, F.~Baldassarri and K.~Kedlaya for useful 
conversations. We would like to thank the referees for 
having read this paper very carefully and provided many 
interesting comments. The present proof of 
Proposition~\ref{prop:opendisc} was suggested by one 
of them and another prompted us to add piecewise 
$\log$-linearity to the statement of our main 
theorem~\ref{thm:continuousandfinite}. We are 
especially grateful to them for this. 


\end{document}